\documentclass[12pt]{amsart}
\usepackage{geometry} 
\usepackage{amssymb}
\usepackage{wasysym} 
\usepackage[utf8]{inputenc} 
\usepackage{textcomp} 
\usepackage{color}
\usepackage{tikz}
\usepackage{tikz-cd}
\usepackage{graphicx}  
\usepackage{flafter}  
\usepackage{bm}  
\usepackage[pdftex,bookmarks,colorlinks,breaklinks]{hyperref}  
\definecolor{dullmagenta}{rgb}{0.4,0,0.4}   
\definecolor{darkblue}{rgb}{0,0,0.4}
\hypersetup{linkcolor=red,citecolor=blue,filecolor=dullmagenta,urlcolor=darkblue} 
\usepackage{pdfsync}  

\DeclareMathOperator{\SE}{SE}
\DeclareMathOperator{\Emb}{Emb}
\DeclareMathOperator{\Diff}{Diff}
\DeclareMathOperator{\vol}{vol}
\DeclareMathOperator{\discrete}{discrete}

\newcommand{\bulk}{\ensuremath{\aquarius}}

\newcommand{\lb}{\ensuremath{\langle}}
\newcommand{\rb}{\ensuremath{\rangle}}
\newcommand{\pder}[2]{\ensuremath{\frac{\partial #1}{\partial #2}}}
\newcommand{\body}{\ensuremath{\mathcal{B}}}

\newtheorem{thm}{Theorem}[section]
\newtheorem{lem}[thm]{Lemma}
\newtheorem{prop}[thm]{Proposition}
\newtheorem{cor}[thm]{Corollary}

\newtheorem{definition}[thm]{Definition}
\newtheorem{example}[thm]{Example}

\begin{document}
\title[The role of $\SE(d)$-reduction for swimming]{The role of $\SE(d)$-reduction for swimming in Stokes and Navier-Stokes fluids}
\author{Henry O. Jacobs}

\maketitle

\begin{abstract}
  Steady swimming appears both periodic and stable.
  These characteristics are the very definition of limit cycles, and so we ask ``Can we view swimming as a limit cycle?''
  In this paper we will not be able to answer this question in full.
  However, we shall find that reduction by $\SE(d)$-symmetry brings us closer.
  Upon performing reduction by symmetry, we will find a stable fixed point which corresponds to a motionless body in stagnant water.
  We will then speculate on the existence of periodic orbits which are ``approximately'' limit cycles in the reduced system.
  When we lift these periodic orbits from the reduced phase space, we obtain dynamically robust relatively periodic orbits wherein each period is related to the previous by an $\SE(d)$ phase.
  Clearly, an $\SE(d)$ phase consisting of nonzero translation and identity rotation means directional swimming, while non-trivial rotations correspond to turning with a constant turning radius.
\end{abstract}

\section{Introduction}
  Many engineers have a justifiable predilection for coordinate-based descriptions of the world.
  However, the use of coordinate-free descriptions is consistently leveraged in Jerry Marsden's work to gain insights which would otherwise have been clouded by the complexities which coordinates bring with them.
  For example, proving anything non-trivial about the inviscid fluid equations
  \begin{align}
  	\partial_t u^i + u^j \partial_j u_i + \partial_i p = 0 \qquad , \qquad \partial_i u^i = 0 \quad , \quad u \in \mathfrak{X}(\mathbb{R}^d) \label{eq:Euler}
  \end{align}
  is notoriously difficult.
  However, with the publication of \cite{Arnold1966} differential geometers were permitted to substitute \eqref{eq:Euler} with a right-trivialized geodesic equation on a Lie group (i.e. an Euler-Poincar\'e equation).
  In particular, if one was willing to use geometry, one could study inviscid fluids \emph{without the need to invoke \eqref{eq:Euler} directly!}
  Only three years later, the proof of local existence-uniqueness was realized by David Ebin and Jerry Marsden using these coordinate-free notions \cite{EbinMarsden1970}.
  
  In studying swimming in the mid-Reynolds regime, one is confronted with coupling a solid body to a Navier-Stokes fluid.
  It is just as true today as it was in 1966 that the Navier-Stokes equations are difficult to work with.
A very modest extension of \cite{Arnold1966} allows us to view fluid-structure problems as forced Lagrangian systems on principal bundles \cite{JaVa2013}.
  In this paper, we will use this geometric characterization of fluid-structure interaction to study swimming in viscous flows.
  We will use these geometric tools to explore the question: \emph{Can we reasonably interpret swimming as a limit cycle?}.
  Unfortunately, we will not be able answer this question in full.
  However, we will be able to clarify the crucial role which $\SE(d)$-symmetry will play in the final answer.
  \emph{A limit-cycle interpretation of swimming is valuable because it would conform with an existing body of knowledge derived from laboratory and computer experiments.
  Moreover, this simple characterization of swimming could be of interest to control engineers who desire to use passive mechanisms to achieve robust behavior with simple open-loop control algorithms.}

\paragraph{\bf Main Contributions}
  We will understand the system consisting of a body immersed in a fluid as a dissipative system evolving on a phase space $P$.  One observes that the system is invariant with respect to the group of isometries of $\mathbb{R}^d$, i.e. the special Euclidean group $\SE(d)$.  This observation suggests that one can describe the system evolving on the quotient manifold $[P] = \frac{P}{\SE(d)}$.  Given this reduction, the main contributions of this paper are:
\begin{itemize}
	\item Under reasonable assumptions on the Lagrangian and the viscous frictions, we will prove the existence of an asymptotically stable point for the dynamics in $[P]$.
	\item We will illustrate how relative limit cycles are produced by exponentially stable equilibria in finite-dimensional dynamical systems under sufficiently small time-periodic perturbations.
	\item For sufficiently small time-periodic internal body forces, we will speculate on the existence of loops in $[P]$ which approximately satisfy the dynamics on $[P]$.
	\item We illustrate how loops in $[P]$ are lifted to paths in $P$, where each period is related to the previous by a rigid rotation and translation.
\end{itemize}

\subsection{Background}
 There exists a substantial body of knowledge in the form of computational and biological experiments which are consistent with the hypothesis that swimming could be interpreted as a limit cycle.
 For example,  experiments involving tethered dead fish immersed in a flow behind a bluff body suggest an ability to passively harvest energy from the surrounding vorticity of the flow.
 The same studies also provide a relevant example of oscillatory behavior as a stable state for an unactuated system \cite{Beal2006}.  Moreover, in living fish, periodic motor neuron actuation has been recorded directly and periodic internal elastic forces have been approximated via linear elasticity models \cite{Shadwick1998}.  Finally, the notion of central pattern generators\footnote{Central patter generators (CPGs) are neural networks which produce time-periodic signals.}, has become widely accepted among biologists studying locomotion \cite{Delcomyn1980,GillnerWallen1985,OttoFriesen1994}. In particular, a central pattern generator for lamprey swimming has been identified and EMG readings have been recorded in-vitro to verify that the swimming mechanism does not rely solely on feedback \cite{WallenWilliams1984} (see Figure \ref{fig:emg}). These experiments and observations from biology suggest that passive mechanisms might play a significant role in understanding swimming.
 
 \begin{figure}[h]
 	\centering
	\includegraphics[width = 3in]{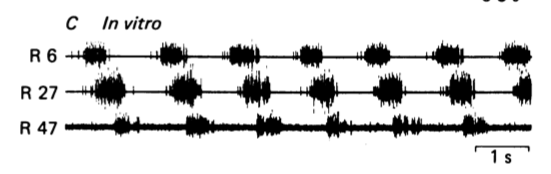}
	\caption{An in-vitro EMG recording of a lamprey spinal chord in ``fictive swimming'' \cite{WallenWilliams1984}.}
	\label{fig:emg}
 \end{figure}

   Additionally, numerical experiments involving rigid bodies with oscillating forces suggest that uniform motion (i.e. flapping flight) is an attracting state for certain pairs of frequencies and Reynolds numbers \cite{Alben_Shelley_2005,Zhang2010}.  Closer to what will be demonstrated here, numerical simulations of a 2-dimensional model of a lamprey at high (but not infinite) Reynolds numbers illustrate swimming as an emergent phenomenon arising asymptotically from time-periodic internal body forces.  Trajectories of this system converge to cyclic behavior after very few oscillations when starting from rest \cite{Tytell2010}.
   A similar study was carried out to understand the difference between periodic control forces and prescribed kinematics in \cite{Wilson2011}.  Here, regular periodic behavior was observed for both.  Moreover, the prescribed kinematic swimmers were unable to swim in the inviscid regime due to time-reversibility, while coherent locomotion was consistently observed for both the forced and prescribed kinematic swimmers at $Re = 70,140,350,560,700$.
   Finally, after the initial submission of this article, a series of numerical experiments to test this ``limit-cycle hypothesis'' were performed for n-linked swimmers. 
   Here, the authors viewed swimming as analogous to the emergence of limit cycles in a forced-damped harmonic oscillator in what they refer to as the ``forced-damped-oscillation framework.''
   The numerical experiments consisted of placing an n-linked chain with an elastic restoring force on the joint angles into a Navier-Stokes fluid using the immersed body method.
   The results consistently suggested that the dynamics admit a stable relative limit cycle \cite{Bhalla2013}. 
   Of course, vorticity shedding plays a fundamental role in the middle and high Reynolds swimming \cite{Vogel}.
   
   \begin{figure}[h]
   	\begin{minipage}{0.45\textwidth}
		\centering
		\includegraphics[width = \textwidth]{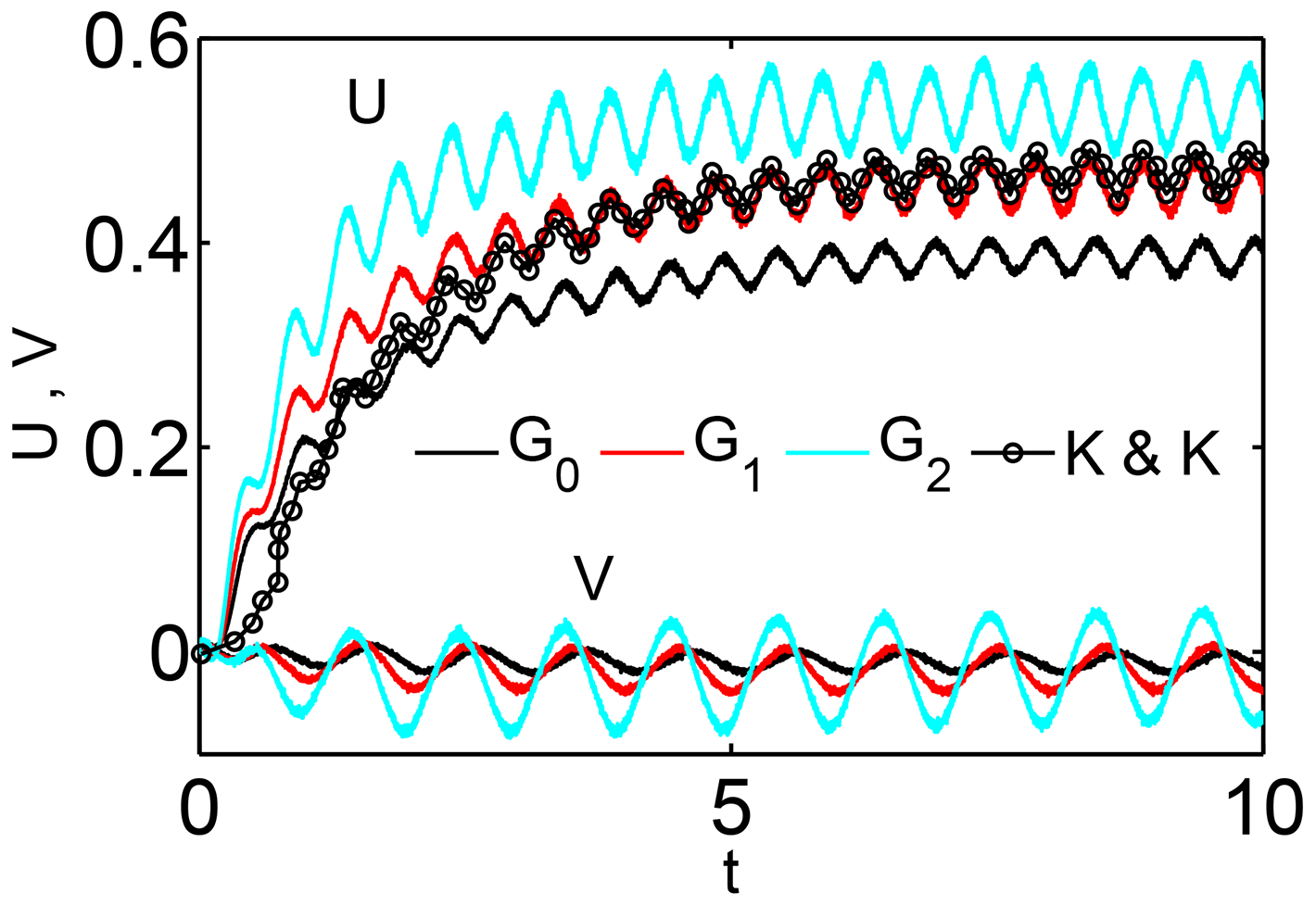}
		\caption{A plot taken from \cite{Bhalla2013} of the horizontal velocity, $U$, and vertical velocity, $V$, of n-linked swimmers with time-periodic internal body forces.}
	\end{minipage}
	\begin{minipage}{0.45\textwidth}
		\centering
		\includegraphics[width = \textwidth]{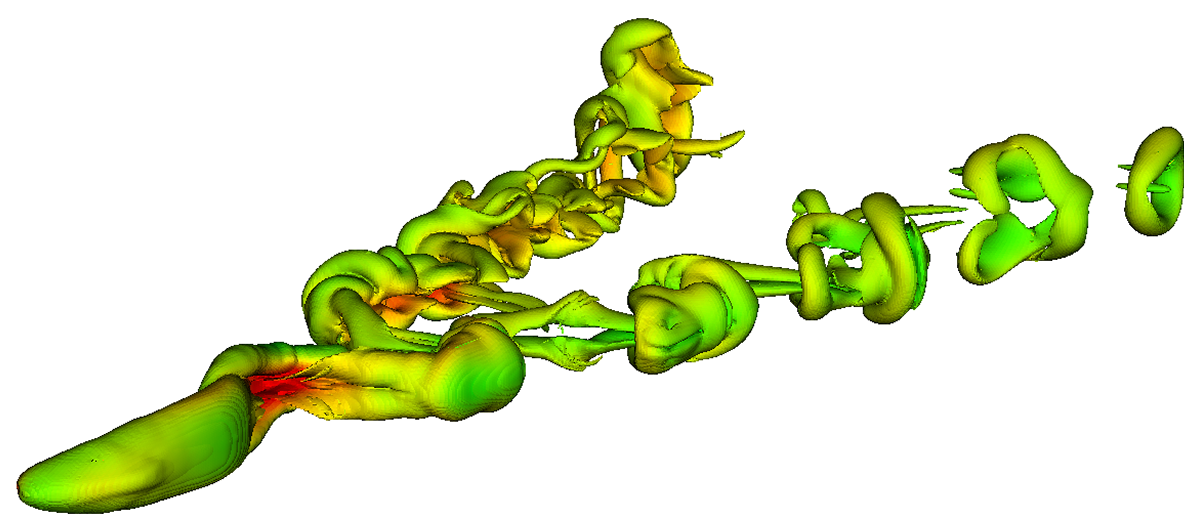}
		\caption{A vorticity isosurface of an n-linked swimmer courtesy of \cite{Bhalla2013}.}
	\end{minipage}
   \end{figure}
   
   In this paper we will approach the problem of fluid-structure interaction in Navier-Stokes fluids as an instance of Lagrangian reduction by symmetry \cite{CeMaRa2001}.
   Recent work based upon a modest generalization of \cite{Arnold1966} has accomplished this reduction by viewing the configuration space of fluid-structure interaction as a $\Diff_{\vol}( \bulk_{b_0})$-principal bundle, 
   where $\Diff_{\vol}( \bulk_{b_0})$ is the diffeomorphism group for the reference-domain of the fluid \cite{JaVa2013}. 
   In particular, the standard equations of motion for a passive body immersed in a Navier-Stokes fluid can be seen as dissipative Lagrange-Poincar\'e equations.
   Following Professor Marsden's tradition of giving credit to Jean le Rond d'Alembert for his formulation of the Lagrange-d'Alembert principle, it would be fair to label the equations of motion for a body immersed in a Navier-Stokes fluid as an instance of \emph{Lagrange-Poincar\'{e}-d'Alembert equations}.
   Just as \cite{Arnold1966} allowed geometers to replace the coordinate-based description of an Euler fluid with an Euler-Poincar\'e equation, \cite{JaVa2013} will serve as a sanity check for us, and allow us to replace the equations for a Navier-Stokes fluid coupled to an elastic solid with a Lagrange-Poincar\'e-d'Alembert equation.
   
   It is worth noting that the constructions to be presented in this paper are different from those typically employed in applications of differential geometry to fluid-structure interaction.
   In the low Reynolds regime, one frequently encounters geometric constructions initially articulated in \cite{ShapereWilczek1989}.
   Similarly, in the potential flow regime, a similar set of constructions was described in \cite{Lewis1986}.
   Both of these constructions lead to a number of insights in aquatic locomotion at extreme Reynolds numbers \cite{Ehlers2011,Kanso2005,Kelly1996,Kelly2000,Koiller1996,Munnier2011}.
  Principal connections are crucial for these constructions, but interpolating between these extreme Reynolds regimes has proven difficult.
  In particular, \emph{there will be absolutely no principal connections in this paper}.

\subsection{Conventions and Notation}
  All objects and morphisms will be assumed to be sufficiently smooth.
  Moreover, we will not address the existence or uniqueness of solutions for fluid-structure systems and all algebraic manipulations will be interpreted formally.
  If $M$ is a smooth manifold then we will denote the tangent bundle by $\tau_M :TM \to M$, and the tangent lift of a map $f:M \to N$ will be denoted $Tf : TM \to TN$.
  The set of vector fields on $M$ will be denoted $\mathfrak{X}(M)$ and the set of time-periodic vector fields on $M$ will be denoted $\mathfrak{X}(M)^{S^1}$.
  A deformation of a vector field $X \in \mathfrak{X}(M)$ is a continuous\footnote{We view $\mathfrak{X}(M)$ as a Fr\'echet vector space.} sequence of vector fields $X_{\varepsilon} \in \mathfrak{X}(M)$ parametrized by a real parameter $\varepsilon \in \mathbb{R}$ which takes values in a neighborhood of $0 \in \mathbb{R}$ and is such that $X_0 = X$.
  Given that $\mathfrak{X}(M)$ is contained in $\mathfrak{X}(M)^{S^1}$, we can consider time-periodic deformations of vector fields as well.
  The flow of a vector-field, $X$, (perhaps time-dependent) will be denoted by $\Phi_t^X$.
  Lastly, given any map $f:M \to N$, the map $f^{-1}:N \to \mathrm{Set}(M)$ is the set-valued map defined by $f^{-1}(n) = \{ m \in M \vert f(m) = n \}$.

\section{Limit cycles} \label{sec:LC}
Let $M$ be a finite-dimensional Riemannian manifold with norm $\| \cdot \| : TM \to \mathbb{R}$.
We can use the norm to define the notion of exponential stability.
Informally, an exponentially stable equilibrium is an equilibrium for which nearby trajectories are attracted to at an exponential rate.
Formally, we say an equilibrium is \emph{exponentially stable} if the spectrum of the linearized system lies \emph{strictly} in the left half of the complex plane.
However, the following (and equivalent) definition will be of greater use.

\begin{definition}
	Let $x^* \in M$ be an equilibrium of the vector field $X \in \mathfrak{X}(M)$.  Let $TX \in \mathfrak{X}(TM)$ be the tangent lift of $X$.  We call $x^*$ an \emph{exponentially stable equilibrium} if there exists a $\lambda < 0$ such that
	\[
		\frac{d}{dt} \| v(t) \| < \lambda \| v(0) \| \quad , \quad \forall t > 0
	\]
	where $v(t)$ is a solution curve of $TX$.
\end{definition}

If one prefers to view exponential stability in terms of flows, we can use the Riemannian distance metric $d: M \times M \to \mathbb{R}$.
Then, an exponentially stable equilibrium $x^* \in M$ of a vector field $X \in \mathfrak{X}(M)$ is an equilibrium where there exists a neighborhood $U \subset M$ containing $x^*$ such that for any integral curve $x(t)$ with $x(0) \in U$ the equation
\[
	d(x(t) , x^*) < e^{\lambda t} d(x(0) , x^*) \quad , \quad \forall t > 0
\]
holds for some $\lambda < 0$.

A special property of exponentially stable equilibria is what some control theorists call \emph{robustness} \cite{ZhouDoyle} and what some dynamical systems theorists call \emph{persistence} \cite{Fenichel1971,Hirsch77,Eldering2013}.
Let $X_\varepsilon \in \mathfrak{X}(M)$ be a deformation of the vector field $X \in \mathfrak{X}(M)$.  Given an exponentially stable point $x^* \in M$ of $X$, we can assert the existence of exponentially stable equilibria of $X_\varepsilon$ for sufficiently small $\varepsilon$.  This robustness of behavior can be vastly generalized by considering normally hyperbolic invariant manifolds.

\begin{definition}[Normally hyperbolic invariant manifold\footnote{This definition was taken from the introduction of \cite{Eldering2013} and is equivalent to the definition used in \cite{Hirsch77}.}]
	Let $N \subset M$ be a compact invariant submanifold of the vector field $X \in \mathfrak{X}(M)$, and let $\Phi_t^X$ be the flow of $X$.
	We call $N$ a \emph{normally hyperbolically invariant manifold} if there exists a $T\Phi_t$-invariant splitting $T_{N}M \equiv TN \oplus E^s \oplus E^u$ and rates $\rho^s < - \rho < 0 < \rho < \rho^u$ such that
	\begin{align}
		\| T\Phi^X_t(v) \| \leq C \cdot e^{ \rho |t|} \| v \| \quad , \quad \forall v \in TN , t \in \mathbb{R} \\
		\| T \Phi^X_t(v)\| \leq C_u \cdot e^{ \rho_u t } \| v \| \quad , \quad \forall v \in E^u ,  t \leq 0 \\
		\| T \Phi^X_t(v)\| \leq C_s \cdot e^{ \rho_s t } \| v \| \quad , \quad \forall v \in E^s , t \geq 0
	\end{align}
	for some constants $C,C_u,C_s > 0$.  If $E^u$ has trivial fibers, then we call $N$ an \emph{exponentially stable invariant manifold}.
\end{definition}
Now that we are equipped with the definition of a normally hyperbolic invariant manifold, we can state the persistence theorem (a.k.a. Fenichel's theorem).
\begin{thm}[see Theorem 1 of \cite{Fenichel1971} or section 4 of \cite{Hirsch77}] \label{thm:persistence}
	Let $X_{\varepsilon} \in \mathfrak{X}(M)$ be a deformation of $X \in \mathfrak{X}(M)$ and let $N \subset M$ be a compact normally hyperbolic invariant manifold of $X$.  Then for sufficiently small $\varepsilon > 0$ there exists a normally hyperbolic invariant manifold, $N_\varepsilon \subset M$ of $X_\varepsilon$ which is  diffeomorphic to $N$ and contained in a neighborhood of $N$.
\end{thm}

We will not need Theorem \ref{thm:persistence} in its full generality because we will only be concerned with a special instance of normally hyperbolic invariant manifolds.  In particular, we will be concerned with exponentially stable limit cycles.

\begin{definition}[Exponentially stable limit cycle]
	An exponentially stable invariant manifold which is homeomorphic to $S^1$ is called an \emph{exponentially stable limit cycle}.
\end{definition}

  We can alternatively define an exponentially stable limit cycle using the distance metric $d:M \times M \to \mathbb{R}$.  Given a periodic trajectory $x^*(t)$, the orbit $\Gamma$ is an exponentially stable limit cycle if there exists a neighborhood $U$ of $\Gamma$ and a contraction rate $\lambda < 0$ such that
  \[
  	d( x(t) , \Gamma) \leq e^{\lambda t} d( x(0) , \Gamma) \quad \forall t > 0
  \]
  for all solution curves $x(t)$ with $x(0) \in U$.  In any case, a direct corollary of Theorem \ref{thm:persistence} is the persistence of exponentially stable limit cycles.  That is to say:

\begin{cor} \label{cor:ESLC}
  Let $\Gamma$ be an exponentially stable limit cycle of $X \in \mathfrak{X}(M)$ and let $X_\varepsilon \in \mathfrak{X}(M)$ be a deformation of $X$.  Then for sufficiently small $\varepsilon > 0$ there exists an exponentially stable limit cycle $\Gamma_\varepsilon$ of $X_\varepsilon$ which is in a neighborhood of $\Gamma$.
\end{cor}

Given a time-periodic vector field $Y \in \mathfrak{X}(M)^{S^1}$, we can consider the \emph{autonomous} vector field on the time-augmented phase space $M \times S^1$ given by $Y \times \partial_\theta \in \mathfrak{X}(M \times S^1)$.
In particular, the vector field $Y \times \partial_\theta$ corresponds to the autonomous dynamical system
  \begin{align*}
  	\dot{\theta} = 1 \qquad , \qquad \dot{x} = Y(x,\theta).
  \end{align*}
  If the vector field $Y \times \partial_\theta$ admits an exponentially stable limit cycle $(x(t) , \theta(t) ) \in M \times S^1$, then $\theta(t) := t \mod 2\pi$ and $x(t)$ is $2\pi$-periodic.
  This observation justifies the following definition.
  
  \begin{definition}
  	Let $Y \in \mathfrak{X}(M)^{S^1}$.  Given a periodic solution curve $x(t) \in M$, we call the orbit $\Gamma := x(S^1)$ a \emph{non-autonomous exponentially stable limit cycle} if $\Gamma \times S^1$ is an exponentially stable limit cycle for $Y \times \partial_\theta$.
  \end{definition}
  Given the definition of a non-autonomous exponentially stable limit cycle, we can specialize Corollary \ref{cor:ESLC} to the case of time-periodic dynamical systems. In particular, we arrive at:
  
  \begin{prop} \label{prop:NAESLC}
  	Let $x^* \in M$ be an exponentially stable equilibrium of $X \in \mathfrak{X}(M)$ and let $X_\varepsilon \in \mathfrak{X}(M)^{S^1}$ be a time-periodic deformation of $X$.  Then for sufficiently small $\varepsilon > 0$ the vector field $X_\varepsilon$ admits a non-autonomous exponentially stable limit cycle in a neighborhood of $x^*$.
  \end{prop}
  
  \begin{proof}
    Because $x^*$ is an exponentially stable equilibrium of $X$, we can see that $(x^* , \theta)$ for $\theta \in S^1$ is a solution curve of $X \times \partial_\theta$ with orbit $\{x^*\} \times S^1$.
    In particular, $\{ x^* \} \times S^1$ is an exponentially stable limit cycle with a contraction rate $\rho_s$ equal to the contraction rate of $x^*$ in the dynamical system defined by $\dot{x} = X(x)$.
    By Corollary \ref{cor:ESLC}, the vector field $X_\varepsilon \times \partial_\theta \in \mathfrak{X}(M \times S^1)$ also exhibits a limit cycle, $(x_\varepsilon(\theta) , \theta)$, in a neighborhood of $\{x^*\} \times S^1$.
    This means that $x_\varepsilon(\theta)$ is a non-autonomous exponentially stable limit cycle for $X_\varepsilon$ in a neighborhood of $x^*$.
  \end{proof}

  The significance of Proposition \ref{prop:NAESLC} is that we can time-periodically deform systems with exponentially stable equilibria to produce non-autonomous exponentially stable limit cycles.

\begin{example}
Consider the equations of motion for a perturbed linear damped mass-spring system,
\begin{align}
	\frac{d}{dt} \begin{bmatrix} x \\ y \end{bmatrix} = \begin{bmatrix} y \\ -x -y \end{bmatrix} + \varepsilon \begin{bmatrix} 0 \\ \sin(t) \end{bmatrix}. \label{eq:example1}
\end{align}
We see that for $\varepsilon = 0$, the system admits an exponentially stable point $(x,\dot{x}) = (0,0)$.
When $\varepsilon > 0$, the non-autonomous limit cycle of Proposition \ref{prop:NAESLC} emerges.
Typical trajectories of the system for $\varepsilon = 0,1$ are shown in Figures \ref{fig:SP} and \ref{fig:LC}.
\end{example}

\begin{figure}[h]
\centering
\begin{minipage}[b]{0.45\textwidth}
\includegraphics[width = \textwidth]{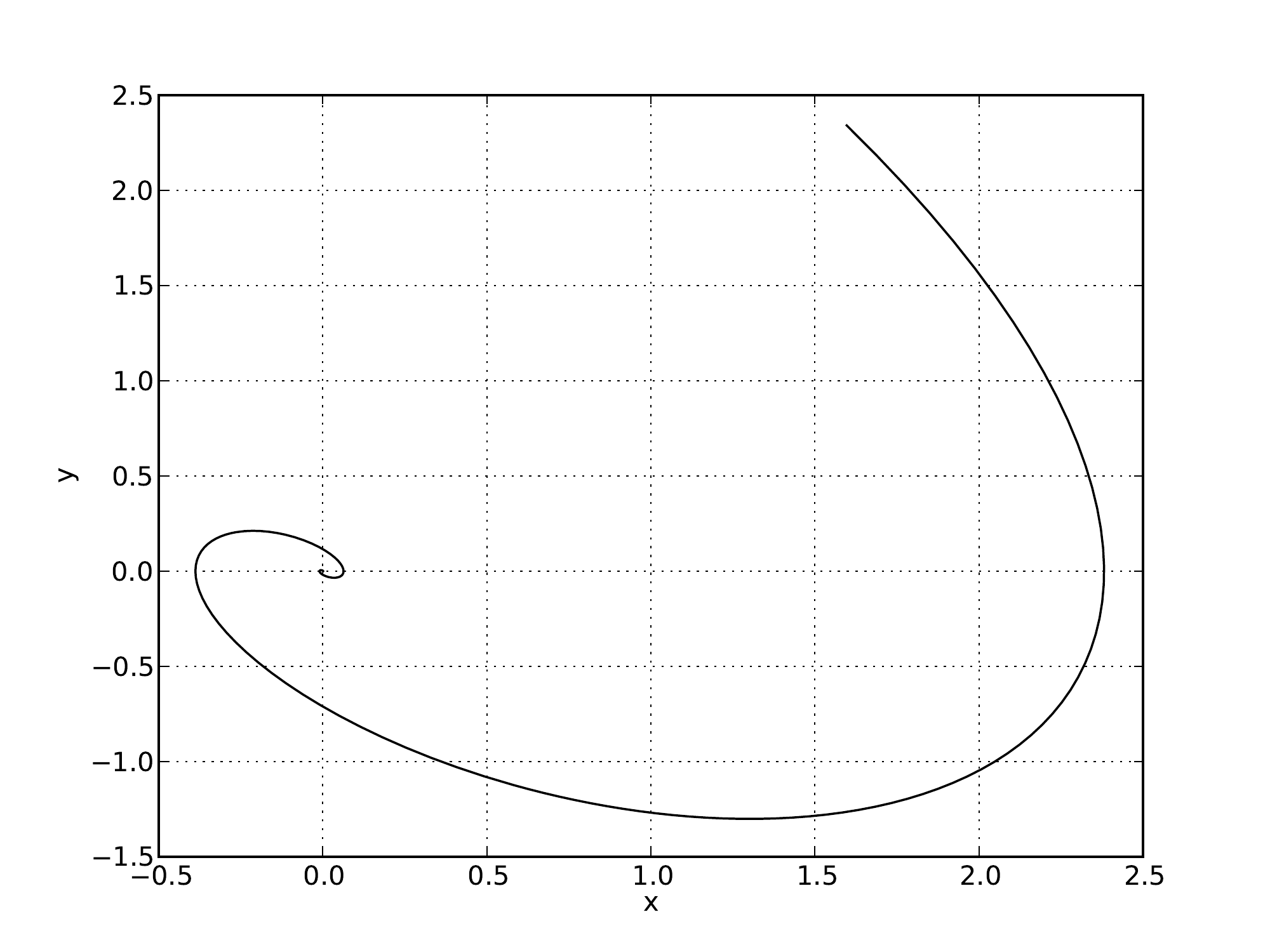}
\caption{A trajectory of \eqref{eq:example1} with $\varepsilon = 0$}
\label{fig:SP}
\end{minipage}
\quad
\begin{minipage}[b]{0.45\textwidth}
\includegraphics[width = \textwidth]{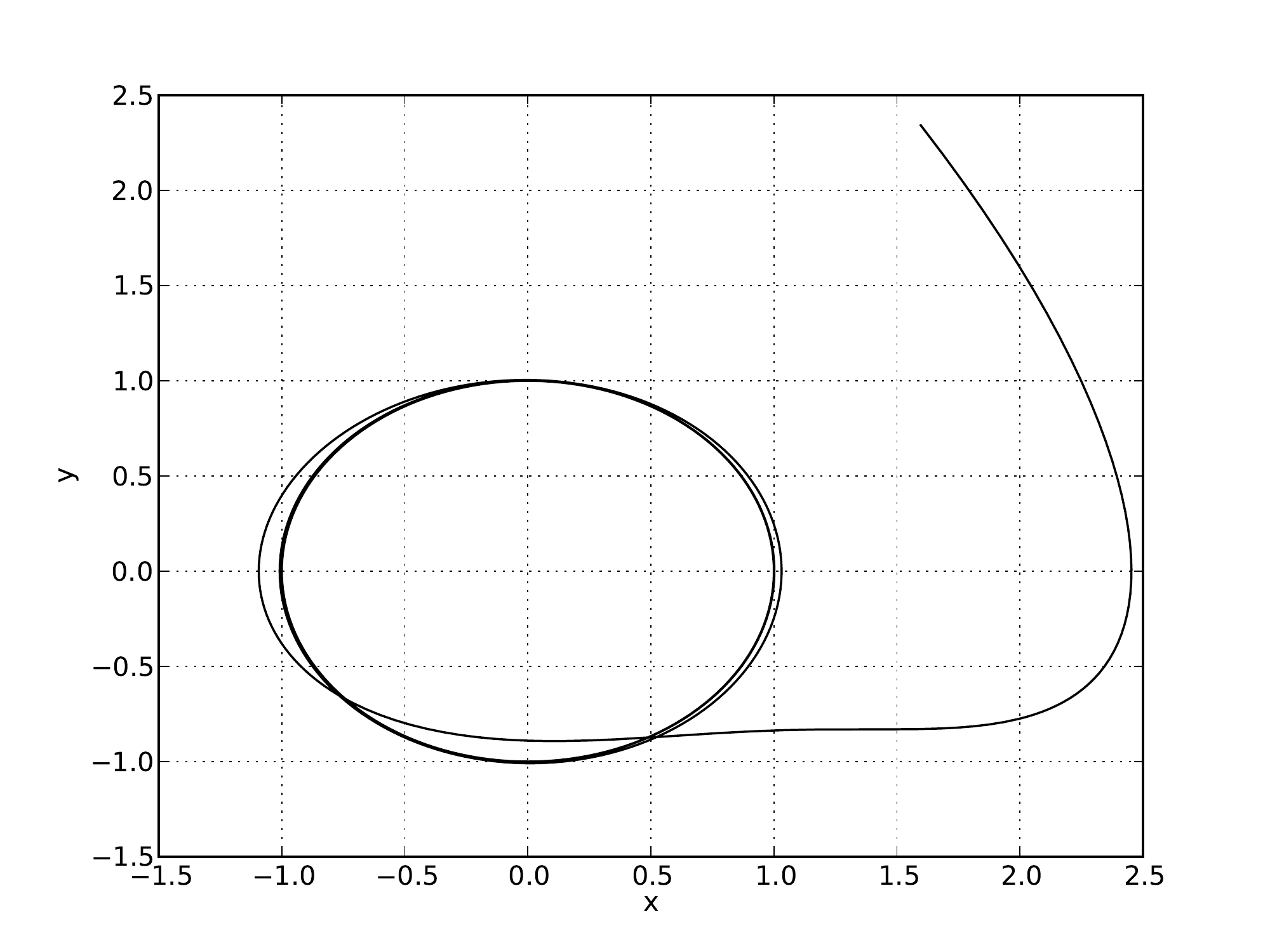}
\caption{A trajectory of \eqref{eq:example1} with $\varepsilon = 1$}
\label{fig:LC}
\end{minipage}
\end{figure}

\section{Relative limit cycles} \label{sec:RLC}
  In this section we will consider dynamical systems with Lie group symmetries.
  Let $G$ be a Lie group which acts on $M$ by a left action.
  The $G$-orbit of a point $x \in M$ is the set $[x] := \{ g \cdot x \vert g \in G \} \equiv G \cdot x$.
  We denote the quotient space by $[M] := \{ [x] : x \in M \}$ and we call the map $\pi: x \in M \mapsto G \cdot x \in [M]$ the quotient projection.
  If the action of $G$ is free and proper, then the quotient projection is a smooth surjection and the triple $(M,[M], \pi)$ is a fiber bundle known as a \emph{principal $G$-bundle} \cite[Proposition 4.1.23]{FOM}.
  
  We now present the natural notions of $G$-invariance for function on $M$.
  Note that for any $[f] \in C^{\infty}([M])$, we can define the smooth function $[f] \circ \pi \in C^{\infty}(M)$.
  Moreover, $[f] \circ \pi$ is $G$-invariant because
  \[
  	[f] \circ \pi(g \cdot x) = [f]( G \cdot (g \cdot x) ) = [f]( G \cdot x) = [f] \circ \pi(x).
  \]
  Conversely, for a $G$-invariant function $f \in C^{\infty}(M)$, we see that $f( G \cdot x) = f(x)$.
  Noting that the left-hand side of this equation involves the application of $f$ to a $G$-orbit, we have apparently found a function $[f] \in C^{\infty}( [M] )$ such that $[f] \circ \pi = f$.
  In other words, the set of $G$-invariant functions on $M$ is identifiable with the set of smooth function on $[M]$.

  This $G$-invariance for functions on $M$ extends to $G$-invariant vector fields.
  We do this by extending the action on $M$ to an action on $TM$ by the tangent lift.
  In particular, the action of $g$ on a point $v = \frac{dx}{dt} \in TM$ is given by
  \[
  	g \cdot v := \left. \frac{d}{dt} \right|_{t=0} (g \cdot x(t) ).
  \]
  In this case, we call $X \in \mathfrak{X}(M)$ a $G$-invariant vector field if 
  \begin{align}
  	g \cdot X(x) = X(g \cdot x) \quad \forall g \in G, x \in M. \label{eq:Gvf}
  \end{align}
  Moreover, the flow $\Phi^X_t$ is $G$-invariant as well.
  For if $X$ is $G$-invariant and $x(t) \in M$ is a solution curve, then
  \[
  	\frac{d}{dt} (g \cdot x(t)) = g \cdot \dot{x}(t) = g \cdot X(x(t)) = X( g \cdot x(t)).
  \]
  Thus, $g \cdot x(t)$ is a solution and so $g \circ \Phi^X_t = \Phi^X_t \circ g$.
  
  \begin{prop}\label{prop:reduced_vf}
  	If $X \in \mathfrak{X}(M)$ is $G$-invariant then there exists a unique vector field $[X] \in \mathfrak{X}([M])$ such that $T\pi \cdot X = [X] \circ \pi$.  Moreover, the flow of $[X]$ is $\pi$-related to the flow of $X$.  In other words, the diagrams
	\[
		\begin{tikzcd}
			M \arrow{r}{X} \arrow{d}{\pi} & TM \arrow{d}{T\pi} \\
			{[M]} \arrow{r}{[X]} & {T[M]}
		\end{tikzcd}
		\qquad , \qquad
		\begin{tikzcd}
			M \arrow{r}{\Phi^X_t} \arrow{d}{\pi} & M \arrow{d}{\pi} \\
			{[M]} \arrow{r}{ \Phi^{[X]}_t} & {[M]}
		\end{tikzcd}
	\]
are commutative.
  \end{prop}

  Given a pair $X \in \mathfrak{X}(M)$ and $[X] \in \mathfrak{X}([M])$ which satisfies \eqref{eq:Gvf}, we call $[X]$ the reduced vector field and $X$ the unreduced vector field.  This correspondence between $X$ and $[X]$ allows us to discuss relative periodicity.

\begin{definition}[relative periodicity]
  Let $X \in \mathfrak{X}(M)$ be a $G$-invariant vector field on the $G$-principal bundle $\pi:M \to [M]$. Let $[X] \in \mathfrak{X}([M])$ be the reduced vector field.  The orbit of a solution curve $x(t)$ of $X$ is called a \emph{relatively periodic orbit} if $\pi(x(t))$ is a periodic orbit of $[X]$.
  \end{definition}

A remarkable characteristic of relative periodic orbits is the following.

\begin{prop} \label{prop:regular}
	Let $X \in \mathfrak{X}(M)$ be a $G$-invariant vector field.  If $x(t)$ is a relative periodic orbit of period $T$, then there exists some $g \in G$ such that $x(T) = g \cdot x(0)$.  Moreover, $x(kT) = g^k \cdot x(0)$ for each $k \in \mathbb{Z}$.
\end{prop}

We call the element $g \in G$ of Proposition \ref{prop:regular} the \emph{phase shift} of the periodic orbit $\pi(x(t))$.
To hint at the relevance of this concept to locomotion, we should mention that if $G = \SE(d)$, the phase shift implies that the system undergoes regular and periodic changes in position and orientation.  We now seek to study exponentially stable manifestations of relative periodicity.  This brings us to the notion of a relative limit cycle.

\begin{definition}
 An orbit $x(t) \in M$ of a $G$-invariant vector field $X \in \mathfrak{X}(M)$ is called a \emph{relative exponentially stable limit cycle} if $\pi(x(t))$ is an exponentially stable limit cycle for the reduced vector field $[X]$.  Finally, if $Y \in \mathfrak{X}(M)^{S^1}$ is $G$-invariant with reduced vector field $[Y] \in \mathfrak{X}_{S^1}([M])$, then we call the orbit of a trajectory $x(t) \in M$ a \emph{non-autonomous exponentially stable relative limit cycle} if $\pi(x(t)) \in [M]$ is a non-autonomous exponentially stable limit cycle.
\end{definition}

\begin{prop} \label{prop:stable}
 Let $X \in \mathfrak{X}(M)$ be $G$-invariant, and let $[X] \in \mathfrak{X}([M])$ be the reduced vector field of $X$.
 Let $\Gamma \subset [M]$ be a limit cycle of $[X]$.  Then there exists an open neighborhood $U$ of $\pi^{-1}( \Gamma) \subset M$ wherein each point is attracted towards a relative limit cycle contained in $\pi^{-1}( \Gamma)$.
\end{prop}

Before we provide the proof of this proposition, it is useful to illustrate the following lemma which relates the distance metric on $M$ with the natural distance metric on $[M]$.

\begin{lem} \label{lem:metric}
  If the Riemannian metric on $M$ is $G$-invariant, then the distance metric $d: M \times M \to \mathbb{R}$ is $G$-invariant as well.  The function on $[M] \times [M]$ given by $[d] ( [x] , [y] ) := d( G \cdot x , G \cdot y)$ is a metric and satisfies the equality $d( x , G \cdot y ) = [d]( [x] , [y])$.
\end{lem}

Equipped with Lemma \ref{lem:metric}, we are now ready to prove Proposition \ref{prop:stable}.

\begin{proof}[proof of Proposition \ref{prop:stable}]
	Let $[U]$ be a neighborhood of $\Gamma$.  Then $U = \pi^{-1}([U]) \subset M$ is an open set as well, since $\pi$ is continuous.  Therefore, given an arbitrary $x \in U$, we see by Lemma \ref{lem:metric} that $\frac{d}{dt} ( d( x , \pi^{-1}(\Gamma) ) ) = \frac{d}{dt}( d(\pi(x),\Gamma) ) < \lambda d(\pi(x),\Gamma) = \lambda d(x, \pi^{-1}( \Gamma) )$.  Thus, the solution is attracted towards $\pi^{-1}(\Gamma)$.  However, $\pi^{-1}(\Gamma)$ is foliated by relative limit cycles.
\end{proof}

Later we will want to see how time-periodic perturbations generate stable and relatively periodic behavior.  This motivates us to state the following proposition.

\begin{prop}\label{prop:NAESRLC}
	Let $X \in \mathfrak{X}(M)$ be $G$-invariant and let $[X] \in \mathfrak{X}([M])$ be the reduced vector field of $X$.
	Let $q^*$ be an exponentially stable equilibria of $[X]$.
	If $X_\varepsilon \in \mathfrak{X}(M)^{S^1}$ is a time-periodic $G$-invariant deformation of $X$,
	then for sufficiently small $\varepsilon > 0$ the vector field $X_\varepsilon$ admits a non-autonomous exponentially stable relative limit cycle.
\end{prop}

\begin{proof}
  Let $[X_\varepsilon] \in \mathfrak{X}( [M] )^{S^1}$ be the reduced vector field corresponding to $X_\varepsilon$ for each $\varepsilon$.
  We can then verify that $[X_\varepsilon]$ is a deformation of $[X]$.
  By Proposition \ref{prop:NAESLC}, the vector field $[X_\varepsilon]$ admits a non-autonomous exponentially stable limit cycle for sufficiently small $\varepsilon$.
  It follows that $X_\varepsilon$ must admit non-autonomous exponentially stable relative limit cycles.
\end{proof}  

\begin{example} \label{ex:3D}
Consider the system on $\mathbb{R}^3$ given by
\begin{align}
	\frac{d}{dt} \begin{bmatrix} x \\ y \\ z \end{bmatrix} = \begin{bmatrix} y \\ -x - y \\ y - x^2 - x y \end{bmatrix} + \varepsilon \begin{bmatrix} 0 \\ \sin(t) \\ \cos(t) \end{bmatrix}. \label{eq:example2}
\end{align}
We see that this system is invariant under translations in the $z$-coordinate.
This is a $(\mathbb{R},+)$-symmetry and the quotient projection is given by $\pi(x,y,z) = (x,y)$.
The reduced vector field is given by equation \eqref{eq:example1}.
By Proposition \ref{prop:NAESRLC}, \eqref{eq:example2} must admit relative limit cycles for sufficiently small $\varepsilon > 0$.
Moreover, as the symmetry of the system is along the $z$-axis, by Proposition \ref{prop:regular} each period of the relative limit cycle should be related to the previous period by a constant vertical shift.
Typical trajectories for $\varepsilon = 0,1$ are depicted in Figures \ref{fig:RSP} and \ref{fig:RLC}
\end{example}

\begin{figure}[h]
\centering
\begin{minipage}[b]{0.45\textwidth}
\includegraphics[clip=true,trim=2cm 1cm 2cm 1cm, width = \textwidth]{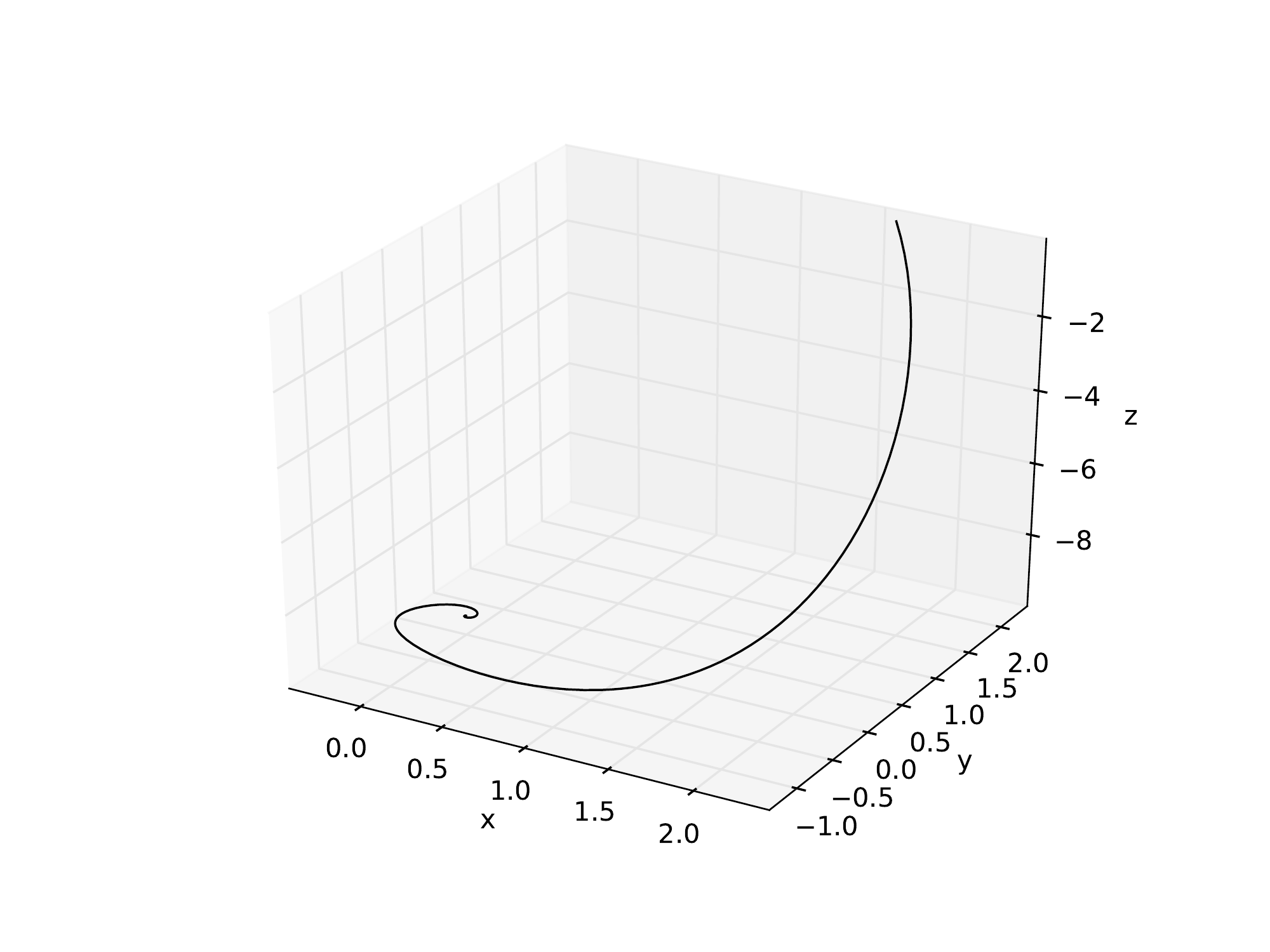}
\caption{A trajectory of \eqref{eq:example2} with $\varepsilon = 0$}
\label{fig:RSP}
\end{minipage}
\quad
\begin{minipage}[b]{0.45\textwidth}
\includegraphics[clip=true,trim=2cm 1cm 2cm 1cm, width = \textwidth]{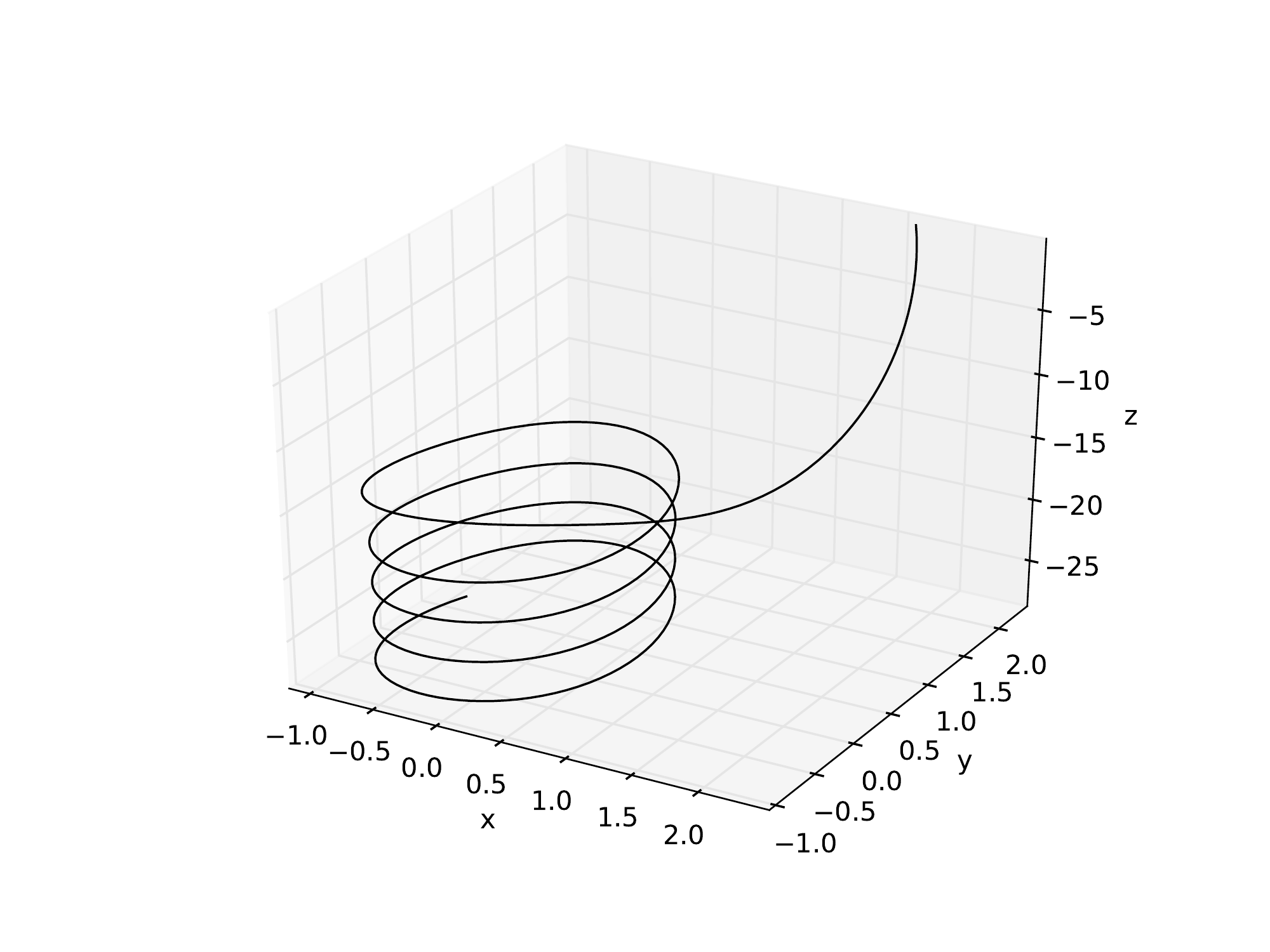}
\caption{A trajectory of \eqref{eq:example2} with $\varepsilon = 1$}
\label{fig:RLC}
\end{minipage}
\end{figure}

Example \ref{ex:3D} illustrates how a $(\mathbb{R},+)$-symmetry lead to a system with a stable non-autonomous relative limit cycle wherein each period was related to the previous by a constant translation along the $z$-axis.
The goal of this article is to characterize swimming as a stable non-autonomous relative limit cycle with respect to an $\SE(d)$-symmetry wherein each period is related to the previous by a constant translation and rotation of space.
In order to do this, we must express fluid-structure problems in a geometric formalism.
In particular, we will follow the constructions of \cite{JaVa2013} to do this, using the Lagrange-d'Alembert formalism.

\section{Lagrange-d'Alembert formalism}  \label{sec:LDA}
In this section, we review the Lagrange-d'Alembert formalism for simple mechanical systems.
If $Q$ is equipped with a Riemanian metric, $\lb \cdot , \cdot \rb_Q: TQ \oplus TQ \to \mathbb{R}$, then it is customary to consider Lagrangians of the form
  \begin{align}
  	L(q,\dot{q}) = \frac{1}{2}  \lb \dot{q} ,\dot{q}\rb_Q - U(q), \label{eq:SM_Lag}
  \end{align}
  where $U: Q \to \mathbb{R}$.  We call a Lagrangian of this form a \emph{simple mechanical Lagrangian}.  For simple mechanical Lagrangians, and external force fields $F:TQ \to T^{\ast}Q$, the Lagrange-D'Alembert equations take the form
  \begin{align}
  	\frac{D \dot{q}}{Dt} = \nabla U(q) + \sharp \left( F(q,\dot{q}) \right), \label{eq:LDA}
  \end{align}
  where $\frac{D}{Dt}$ is the Levi-Cevita covariant derivative, $\nabla U$ is the gradient of $U$, and $\sharp : T^{\ast}Q \to TQ$ is the sharp operator induced by the Riemannian metric \cite[Proposition 3.7.4]{FOM}.
  It is notable that \eqref{eq:LDA} is equivalent to the Lagrange-d'Alembert variational principle
\[
	\delta \int_0^T L(q,\dot{q}) dt = \int_0^{T}{ \lb F(q,\dot{q}) , \delta q \rb dt }
\]
with respect to variations $\delta q$ with fixed end points \cite[Chapter 7]{MandS}.
We denote the vector field associated to \eqref{eq:LDA} by $X_{TQ} \in \mathfrak{X}(TQ)$, and its flow is given by $\Phi^{TQ}_t$.

%
%

\section{Fluid-structure interaction}  \label{sec:passive}
In this section, we will place fluid structure-interactions in the Lagrange-d'Alembert formalism.
Specifically, we will understand a body immersed in a fluid as a simple mechanical Lagrangian system with a dissipative force field, in the sense of \eqref{eq:LDA}.
This is commonly referred to as the ``material description'' in fluid mechanics.
Moreover, we will reduce the system by a particle relabeling symmetry, so that the fluid is described in the ``spatial description'' via the Navier-Stokes equations.
Finally, we will identify `frame-invariance' (a.k.a objectivity) as a left $\SE(d)$-symmetry.

\subsection{Navier stokes fluids in the Lagrange-d'Alembert formalism} \label{sec:passive fluids}
  In this paper, we seek to understand swimming in the mid-Reynolds regime.
  Specifically this entails invoking the Navier-Stokes equations with non-zero viscosity.
  It was discovered in \cite{Arnold1966} that the Navier-Stokes equations with zero viscosity could be handled in the Euler-Poincar\'e formalism.
  Moreover, it is mentioned in \cite[Chapter 1, section 12]{ArKh1992} that the Navier-Stokes equations can be viewed in this framework with the simple addition of a dissipative force.
  In this section, we will describe this formulation of the Navier-Stokes equations.
  
  Consider the manifold $\mathbb{R}^d$ with the standard flat metric and volume form $dx = dx^1 \wedge \cdots \wedge dx^d$.
  One can consider the infinite-dimensional Lie group of volume-preserving diffeomorphisms, $\Diff_{\vol} (\mathbb{R}^d)$, where the group multiplication is simply the composition of diffeomorphisms.\footnote{This is a pseudo Lie group.  We will assume that all diffeomorphisms approach the identity as $\|x\| \to \infty$ sufficiently rapidly for all computations to make sense.  In particular, the existence of a Hodge-decomposition for our space is important.  Sufficient conditions for our purposes are provided in \cite{Cantor1975} and \cite{Troyanov2009}.}
  The configuration of a fluid flowing on $\mathbb{R}^d$ relative to some reference configuration is described by an element $\varphi \in \Diff_{\vol}(\mathbb{R}^d)$.
  Given a curve $\varphi_t \in \Diff_{\vol}(\mathbb{R}^d)$, one can differentiate it to obtain a tangent vector $\dot{\varphi}_t = \frac{d}{dt} \varphi_t \in T \Diff_{\vol}(\mathbb{R}^3)$.
  One can interpret $\dot{\varphi}$ as a map from $\mathbb{R}^d$ to $T\mathbb{R}^d$ by the natural definition $\dot{\varphi}(x) = \frac{d}{dt} \varphi_t(x)$.
  Therefore, a tangent vector, $\dot{\varphi} \in T \Diff_{\vol}(\mathbb{R}^d)$, over a diffeomorphism $\varphi \in \Diff_{\vol}(\mathbb{R}^d)$ is simply the smooth map $\dot{\varphi}: \mathbb{R}^d \to T\mathbb{R}^d$, such that $\tau_{\mathbb{R}^d} \circ \dot{\varphi} = \varphi$ where $\tau_{\mathbb{R}^d} : T\mathbb{R}^d \to \mathbb{R}^d$ is the tangent bundle projection.
  Moreover, $\dot{\varphi} \circ \varphi^{-1}$ is a smooth divergence-free vector field on $\mathbb{R}^d$.
  We call $\dot{\varphi}$ the \emph{material} representation of the velocity, while $\dot{\varphi} \circ \varphi^{-1} \in \mathfrak{X}_{\vol}(\mathbb{R}^d)$ is the spatial representation.
 The Lagrangian, $L: T( \Diff_{\vol}(\mathbb{R}^d) ) \to \mathbb{R}$, is the kinetic energy of the fluid,
  \[
  	L( \varphi, \dot{\varphi}) := \frac{1}{2} \int_{\mathbb{R}^d}{ \| \dot{\varphi}(x) \|^2 dx}.
  \]
  One can derive the Euler-Lagrange equations on $\Diff_{\vol}(\mathbb{R}^d)$ with respect to the Lagrangian $L$ to obtain the equations of motion for an ideal fluid.  However, this Lagrangian exhibits a symmetry.
  \begin{prop}[\cite{Arnold1966}]
  	The Lagrangian $L$ is symmetric with respect to the right action $\Diff_{\vol}(\mathbb{R}^d)$ on $T \Diff_{\vol}(\mathbb{R}^d)$.
\end{prop}

Moreover, it is simple to verify the the proposition:

\begin{prop} \label{prop:MC}
	The action of $\Diff_{\vol}( \mathbb{R}^d )$ on $T \Diff_{\vol}( \mathbb{R}^d)$ is free and proper.
	The quotient space $T \Diff_{\vol}(\mathbb{R}^d) / \Diff_{\vol}( \mathbb{R}^d) = \mathfrak{X}_{\rm div}(\mathbb{R}^d)$ and the quotient projection is the right Maurer-Cartan form,
	\[
		\rho: (\varphi, \dot{\varphi}) \in T \Diff( \mathbb{R}^d) \mapsto u = \dot{\varphi} \circ \varphi^{-1} \in \mathfrak{X}_{\rm div}( \mathbb{R}^d).
	\]
\end{prop}

   This symmetry is referred to as the \emph{particle relabeling symmetry}.
   As a result of this symmetry, Proposition \ref{prop:reduced_vf} suggests that we can write equations of motion on $\mathfrak{X}_{\rm div}( \mathbb{R}^d )$.
   It was the discovery of \cite{Arnold1966} that these equations could be written as
  \[
  	\partial_t u + u \cdot \nabla u = - \nabla p \quad , \quad \mathrm{div}(u) = 0,
  \]
  which one will recognize as the inviscid fluid equations.   Moreover, if we define the linear map, $f_\mu: \mathfrak{X}_{\vol}(\mathbb{R}^d) \to \mathfrak{X}_{\vol}^*(\mathbb{R}^d)$ given by
  \[
  	\lb f_\mu(u) , w \rb = \mu \int_{\mathbb{R}^d }{  \Delta u(x)  \cdot w(x)  dx},
  \]
  then we derive the Lagrange-D'Alembert equations by lifting $f_\mu$ (via the right Maurer-Cartan form) to obtain a force field $F: T (\Diff_{\vol}(\mathbb{R}^3)) \to T^{\ast}( \Diff_{\vol}(\mathbb{R}^3))$.
  If we do this, then reduction by $\Diff_{\vol}(\mathbb{R}^3)$ yields a spatial velocity field, $u(t)$, which satisfies the Navier-Stokes equations
  \[
  	\partial_t u + u \cdot \nabla u = - \nabla p - \mu \Delta u \quad , \quad \mathrm{div}(u) = 0.
  \]

\subsection{Solids} \label{sec:passive solids}
  Let $\body$ be a compact manifold with boundary $\partial \body$ and volume form $d\vol_{\body}$.  Let $\Emb(\body)$ denote the set of embeddings of $\body$ into $\mathbb{R}^d$. Finally, let $\SE(d)$ denote the set of isometries of $\mathbb{R}^d$.
    
  We view each $b \in \Emb(\body)$ as a map $b:\body \hookrightarrow \mathbb{R}^d$, while viewing $z \in \SE(d)$ as a map $z : \mathbb{R}^d \to \mathbb{R}^d$.
  We can compose these maps to obtain a new map $z \circ b : \body \hookrightarrow \mathbb{R}^d$, which itself embeds $\body$ into $\mathbb{R}^d$.
  That is to say, the assignment $b \mapsto z \circ b$ is a left action of $\SE(d)$ on $\Emb(\body)$.
  It is elementary to observe that this action is free and proper, and makes $\Emb(\body)$ into an $\SE(d)$-principal bundle.
  The configuration manifold for the body is given by a $\SE(d)$-invariant submanifold $B \subset \Emb(\body)$ (possibly finite-dimensional).
  Therefore, the quotient space $[B] = \frac{B}{\SE(d)}$ is a smooth manifold and $\pi^{B}_{[B]}: B \to [B]$ is a $\SE(d)$-principal bundle as well.
  We call $[B]$ the \emph{shape-space}, following \cite{MaMoRa1990}.
  
  The Lagrangian for the body, $L_B: TB \to \mathbb{R}$, will be that of a simple mechanical system.
  The reduced-potential energy will be given by a function $[U] : [B] \to \mathbb{R}$,
  and the potential energy is defined as $U := [U] \circ \pi^B_{[B]}$.
  Equivalently, we may define $U:B \to \mathbb{R}$ first, with the assumption that we choose something which is $\SE(d)$-invariant.

  To define the kinetic energy, we must first understand the tangent bundle $TB \subset T \Emb(\body)$.
 By applying the dynamic definition of tangent vectors, we can derive that a $(b,\dot{b}) \in T \Emb(\body)$ must be a pair of maps, $b \in \Emb(\body)$ and $\dot{b} : \body \hookrightarrow T \mathbb{R}^d$, such that $\dot{b}(x)$ is a vector over $b(x)$ for all $x \in \body$.
  Moreover, a $(b,\dot{b}) \in TB$ is an element of $T \Emb(\body)$ tangential to $B \subset \Emb(\body)$.
  We see that for each $z \in \SE(d)$, we can consider the map $Tz : T \mathbb{R}^d \to T \mathbb{R}^d$, and we define the action of $z$ on $TB$ by the assignment $(b,\dot{b}) \in TB \mapsto (z \circ b , Tz \circ \dot{b}) \in TB$.
  This defines a free and proper left $\SE(d)$ action on $TB$ so that $TB$ is an $\SE(d)$-principal bundle.
  We will assume the existence of an $\SE(d)$-invariant Riemannian metric $\lb \cdot , \cdot \rb_{\body} : TB \oplus TB \to \mathbb{R}$,
  and that the kinetic energy is $K(b,\dot{b}) = \frac{1}{2} \langle (b,\dot{b}) , (b,\dot{b}) \rangle_{B}$.
  
  Finally, without any dissipation, our solid body could ``jiggle'' forever due to conservation of energy.  To amend this, we will include a dissipative force given by a fiber-bundle map, $F_{\body}: T [ B] \to T^{\ast} [B]$, such that the storage function
\begin{align}
	\lb F_{\body}( v_{[b]} ) , v_{[b]} \rb : T[ B] \to \mathbb{R} \label{eq:concavity}
\end{align}
is convex on each fiber of $T[B]$ and reaches a maximum at zero where it vanishes.
For example, a negative definite quadratic form would be admissible.
Such a force has the effect of dampening the rate of change in the shape of the body, but it will not dampen motions induced by the action of $\SE(d)$.  In other words, we assume that a jiggling body eventually comes to rest with some shape $s_{\min} \in [B]$ by the dissipation of energy.

  \begin{example} \label{ex:two_link}
 Consider a two-link body in $\mathbb{R}^2$.
  The configuration manifold $B$ consists of rigid embeddings of the two links into $\mathbb{R}^2$ such that the embeddings respect the constraint that the links are joined at the hinge (see Figure \ref{fig:two_link}).
  In particular, $B$ is isomorphic to $S^1 \times S^1 \times \mathbb{R}^2$ if we let the tuple $(\phi_1,\phi_2,x,y ) \in S^1 \times S^1 \times \mathbb{R}^2$ denote a configuration where $\phi_1,\phi_2 \in S^1$ are the angles between the links and the $x$-axis, while $(x,y) \in \mathbb{R}^2$ is the location of the hinge. Under this identification, the action of an element $(\theta, X,Y) \in \SE(2)$ on $(\phi_1,\phi_2,x,y) \in B$ is given by 
  \[
  	(\theta , X , Y) \cdot \begin{pmatrix} \phi_1 \\ \phi_2 \\ x \\ y \end{pmatrix} = \begin{pmatrix} \theta + \phi_1 \\ \theta+ \phi_2 \\ \cos(\theta) x - \sin(\theta)y + X \\ \sin(\theta)x + \cos(\theta) y + Y \end{pmatrix}.
  \]
  Under this action, we find that the shape space is $[B] = S^1$ and that the quotient projection from $B$ to $[B]$ is given by $\pi^B_{[B]}( \phi_1, \phi_2,x,y) = \phi_1 - \phi_2$.
  In other words, the shape of the body is described by the interior angle of the hinge.
  Finally, we may consider a potential energy derived from a linear spring between the hinges given by $U( \phi_1, \phi_2, x,y) = \frac{k}{2} ( \phi_1 - \phi_2 - \bar{\theta} )^{2}$ for some constant equilibrium interior angle $\bar{\theta} \in S^1$.
  It should be evident that this potential energy is $\SE(2)$-invariant.  The kinetic energy of the $i^{\rm th}$ body is
  \[
  	K_i = \frac{I_i}{2} \dot{\phi}_i^2 + \frac{M_i}{2} \left( [\dot{x} - \sin(\phi_i) \dot{\phi_i}]^2 + [\dot{y} + \cos(\phi_i) \dot{\phi_i} ]^2 \right),
  \]
  where $M_i$ and $I_i$ are the mass and rotational inertial of the $i^{\rm th}$ body, respectively.  The Lagrangian is therefore $L_B = K_1 + K_2 - U$.
  Lastly, the force $F_B = \dot{\phi}_2 d\phi_2 - \dot{\phi}_1 d \phi_1$ provides an $\SE(2)$-invariant elastic friction force.
  The effect of $F_B$ is to dampen changes in the the interior angle $\theta = \phi_2 - \phi_1$.
  In particular, $\theta$ parametrizes the shape space of this body, and so $F_B$ can be said to dampen changes in shape.
\end{example}

\begin{figure}[h]
   \centering
   \includegraphics[width=3in]{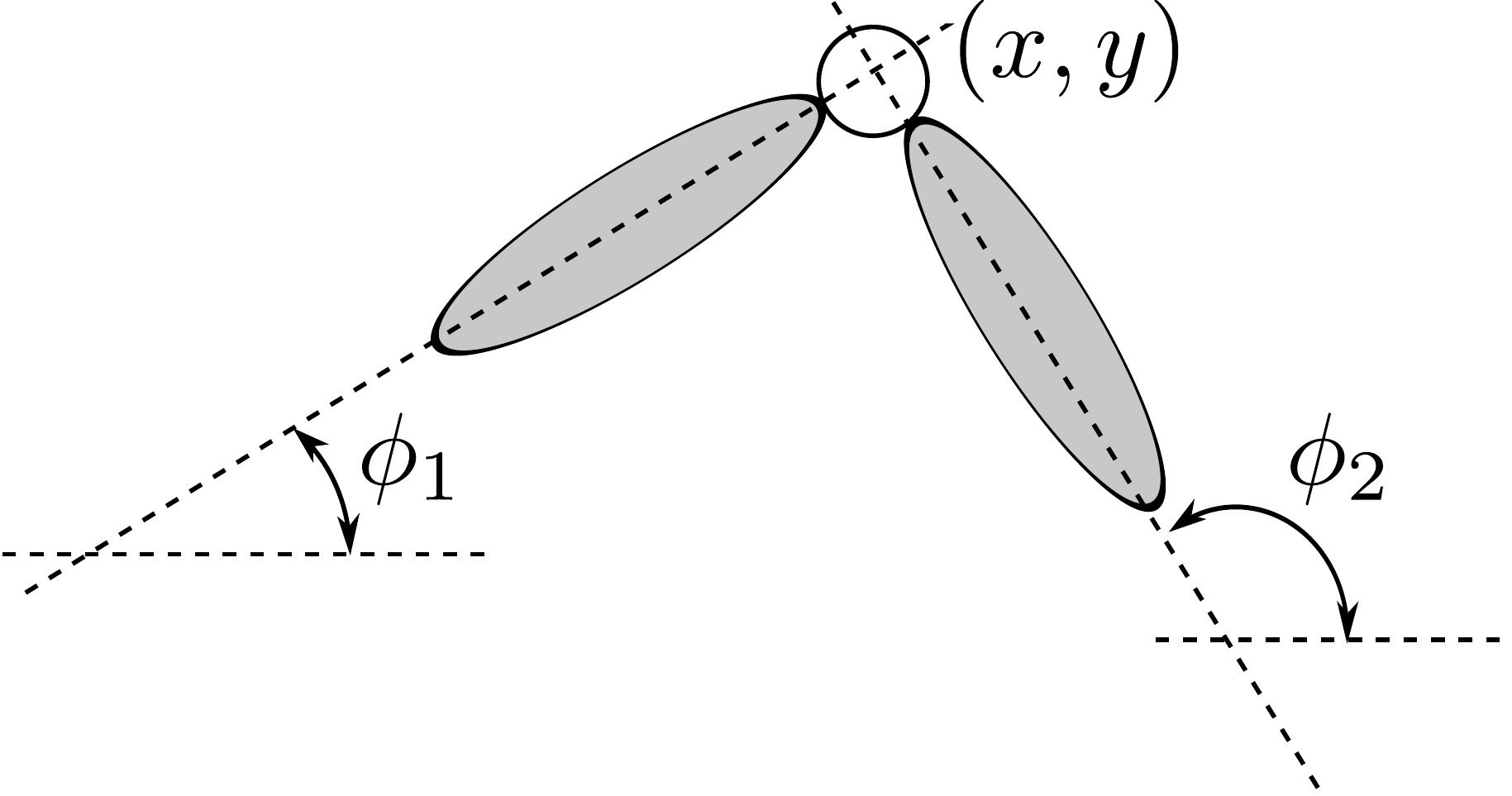} 
   \caption{A diagram of the swimmer from Example \ref{ex:two_link}.}
   \label{fig:two_link}
\end{figure}

  \begin{example}
  The theory of linear elasticity assumes $\body$ to be a Riemannian manifold with a mass density $\rho \in \bigwedge^n (\body)$ and metric $\langle \cdot , \cdot \rangle_{\body}: T\body \oplus T\body \to \mathbb{R}$.
  Here, the configuration manifold is $B = \Emb(\body)$ and the potential energy is
  \[
  	U(b) = \frac{1}{2} \int_{b(\body)}{ \mathrm{trace}\left( [I - C_b]^T \cdot [I - C_b] \right) b_*d\vol_{\body}},
  \]
  where $C_b$ is the push-forward of the metric $\lb \cdot , \cdot \rb_{\body}$ by $b:\body \hookrightarrow \mathbb{R}^d$, a.k.a. the \emph{right Cauchy-Green strain tensor}.
  The $\SE(d)$-invariant kinetic energy, $K:TB \to \mathbb{R}$, is given by
  \[
  	K(b,\dot{b}) = \frac{1}{2} \int_{\body} \| \dot{b}(x) \|^2 \rho(x) dx.
  \]
  This Lagrangian yields the standard model of linear elasticity and is known to be $\SE(d)$-invariant, a.k.a. \emph{objective} \cite{MFOE}.
  As we can not easily coordinatize $B$ in this example, we cannot expect to easily obtain a concrete description of the shape space, $[B]$.
  Nonetheless, by the $\SE(d)$-invariance of $U$, there must exist a function $[U] : [B] \to \mathbb{R}$ such that $U = [U] \circ \pi^B_{[B]}$.
  \end{example}
  The above examples are merely instances of possible models we may choose for the body.
  Identifying a physical model of the solid body is a necessary precondition for understanding the effect of internal body forces on the system.
  In particular, this is the approach taken in \cite{Tytell2010} and \cite{Bhalla2013}.
  To quote \cite{Tytell2010}, ``the motion of the body emerges as a balance between internal muscular force and external fluid forces.''
  The emphasis on the importance of the internal mechanics of the swimming body can become fairly sophisticated.
  These sophisticated solid-mechanical concerns can be important for understanding the role of passive mechanisms in biomechanics.
  For example, fibered structures can exhibit a ``counter-bend phenomena'' in which an increased curvature in one region of a structure yields a decrease elsewhere in ways which aide swimming \cite{Gadelha2013}.
  These advanced topics will not be addressed here, but we recall them only to put this work in a proper context.

\subsection{Fluid-solid interaction} \label{sec:passive fsi}
  Let $\body$, $B$, $L_{\body}$, $F_{\body}$ be as described in the previous section.  Given an embedding $b \in B$, let $\bulk_{b}$ denote the set
  \[
  	\bulk_{b} = \mathrm{closure} \left\{ \mathbb{R}^d \backslash b\left( \body \right) \right\}.
  \]
  The set $\bulk_{b}$ is the region which will be occupied by the fluid given the embedding of the body $b$.  If the body configuration is given by $b_0 \in B$ at time $0$ and $b \in B$ at time $t$, then the configuration of the fluid is given by a volume-preserving diffeomorphism from $\bulk_{b_0}$ to $\bulk_{b}$, i.e. an element of $\Diff_{\vol}\left( \bulk_{b_0}, \bulk_{b} \right)$.  Given a reference configuration $b_0 \in B$ for the body, we define the configuration manifold as
  \begin{align*}
  	Q := \{ (b , \varphi) \quad \vert \quad & b \in B, \varphi \in \Diff_{\vol}\left( \bulk_{b_0}, \bulk_{b} \right) \}.
  \end{align*}
  One should note that the manifold $Q$ has some extra structure.  In particular, the Lie group $G := \Diff_{\vol}( \bulk_{b_0} )$ represents the symmetry group for the set of particle labels, and acts on $Q$ on the right by sending
  \[
  	(b,\varphi) \in Q \mapsto (b,\varphi \circ \psi) \in Q
  \]
  for each $\psi \in G$ and $(b,\varphi) \in Q$.
  Given this action, the following proposition is self-evident.


  \begin{prop} \label{prop:Q}
  	The projection $\pi^Q_B :Q \to B$ defined by $\pi^{Q}_{B}(b,\varphi) = b$ makes $Q$ into a principal $G$-bundle over $B$.
  \end{prop}
  
  Now we must define the Lagrangian.  To do this, it is useful to note that the system should be invariant with respect to particle relablings of the fluid, and so the Lagrangian should be invariant with respect to the right action of $G$ on $TQ$ given by
  \[
  	(b,\dot{b},\varphi,\dot{\varphi}) \in TQ \mapsto (b,\dot{b}, \varphi \circ \psi , \dot{\varphi} \circ \psi ) \in TQ
\]
 for each $\psi \in G$.  As a result, we can define a Lagrangian on the quotient space $TQ / G$.  Incidentally, this quotient space is much closer to the space typically encountered in fluid-structure interaction.

\begin{prop}[Proposition 2.2 of \cite{JaVa2013}] \label{prop:TQ/G}
  The quotient space $TQ / G$ can be identified with the set
 \begin{align}
 	P := \{ (b,\dot{b} , u) \quad \vert \quad & (b,\dot{b}) \in T B , \nonumber \\
	& u \in \mathfrak{X}_{\rm div}( {\bulk_{b}} ) , \nonumber \\
	& u(b(x)) = \dot{b}(x) , \forall x \in \partial \body \}.
 \end{align}
  Under this identification, the quotient map $\pi^{TQ}_{/G}: TQ \to P$ is given by $\pi_{/G}(b,\dot{b},\varphi,\dot{\varphi}) = (b,\dot{b}, \dot{\varphi} \circ \varphi^{-1})$.
 Moreover, $P$ is naturally equipped with the bundle projection $\tau( b,\dot{b},u) = b$ and the vector bundle structure $(b,\dot{b}_1,u_1) + (b,\dot{b}_2, u_2) = (b, \dot{b}_1 + \dot{b}_2, u_1+u_2)$,
 for all $(b,\dot{b}_1,u_1),(b,\dot{b}_2,u_2) \in \tau^{-1}(b)$ and all $ b \in B$.
\end{prop}
\begin{proof}
 Observe that $\pi^{TQ}_{/G}( v \circ \psi) = \pi^{TQ}_{ / G}( v )$ for all $\psi \in G$ and $v \in TQ$.  Therefore, $\pi^{TQ}_{/G}$ maps the coset $v \cdot G$ to a single element of $P$.
 Conversely, given an element $(b,\dot{b}, u) \in P$, we see that $(\pi^{TQ}_{/G})^{-1}(b,\dot{b},u)$ is the set of element in $(b,\dot{b},\varphi , \dot{\varphi}) \in TQ$ such that $u = \dot{\varphi} \circ \varphi^{-1}$.
 However, this set of elements is just the coset $v \cdot G$, where $v$ is any element such that $\pi^{TQ}_{/G}(v) = (b,\dot{b},u)$.  Thus, $\pi^{TQ}_{/G}$ induces an isomorphism between $TQ/G$ and $P$.
 Additionally, we can check that $\pi^{TQ}_{/G}(v+w) = \pi^{TQ}_{/G}(v) + \pi^{TQ}_{/G}(w)$ and $\tau( \pi_{/G}( v ) ) = \tau_{Q}(v) \cdot G$.
 Therefore, the desired vector bundle structure is inherited by $P$ as well, and $\pi_{/G}$ becomes a vector bundle morphism.
 Finally, the map $\rho(b,\dot{b},u) = (b,\dot{b})$ is merely the map $T\pi_B^Q : TQ \to TB$ modulo the action of $G$.  That is to say, $\rho \circ \pi^{TQ}_{/G} = T \pi^Q_B$.  This equation makes $\rho$ well-defined because $\pi^Q_G$ is $G$-invariant.
\end{proof}

As a guide for the reader, we provide the following commutative diagram.
\[
	\begin{tikzcd}
		TQ \arrow{rr}{T\pi^{Q}_{/G} } \arrow{dr}{ \pi^{TQ}_{/G}} \arrow{dd}{\tau_Q} & & TB \arrow{dd}{\tau_B} \\
		& P \arrow{ur}{\rho} \arrow{dr}{\tau}& \\
		Q \arrow{rr}{\pi^Q_{/G}} & & B
	\end{tikzcd}
	\qquad \qquad
	\begin{tikzcd}
		(b,\dot{b},\varphi,\dot{\varphi}) \arrow[mapsto]{rr}{T\pi^Q_{/G} } \arrow[mapsto]{dr}{ \pi^{TQ}_{/G}} \arrow[mapsto]{dd}{\tau_Q} & & (b,\dot{b}) \arrow[mapsto]{dd}{\tau_B} \\
		& (b,\dot{b}, \dot{\varphi} \circ \varphi^{-1}) \arrow[mapsto]{ur}{\rho} \arrow{dr}{\tau}& \\
		(b,\varphi) \arrow[mapsto]{rr}{\pi^Q_{/G}} & & b
	\end{tikzcd}
\]

  Note that the fluid velocity component $u \in \mathfrak{X}_{\rm div}( \bulk_b )$ for a $(b,\dot{b},u) \in P$ may point in directions transverse to the boundary of the fluid domain $\bulk_{b}$.
  This reflects the fact that the boundary is time-dependent.  The condition $\dot{b}(x) = u(b(x))$ on the boundary states that the boundary of the body moves with the fluid, and is the mathematical description of the no-slip condition.

  We now define the reduced Lagrangian $\ell: P \to \mathbb{R}$ by
  \[
  	\ell (b,\dot{b},u) = L_{\body}(b,\dot{b}) + \frac{1}{2} \int_{\bulk_{b} }{ \| u(x) \|^2 dx }.
  \]
  This induces the standard Lagrangian
  \begin{align}
    L:= l \circ \pi^{TQ}_{/G}: TQ \to \mathbb{R} \label{eq:total_Lagrangian},
   \end{align} which is a simple mechanical Lagrangian consisting of the kinetic energy of the fluid minus the potential energy of the body described in \S \ref{sec:passive solids}.  Moreover, $L$ is $G$-invariant by construction.
  
  Additionally, we wish to add a viscous force on the fluid, $F_{\mu} : TQ \to T^{\ast}Q$. Given a coefficient of viscosity, $\mu$, we can define the reduced viscous friction force field $f_{\mu}: P \to P^{\ast}$ by
  \[
  	\lb f_{\mu}( b, v_b, u) , (b, w_b , w) \rb = \mu \int_{\bulk_{b}}{  \Delta u (x)  \cdot w(x)  d x },
  \]
  and define the unreduced force $F_{\mu}: TQ \to T^{\ast}Q$ by
  \[
  	\lb F_{\mu}( v ) , w \rb = \lb f_{\mu}( \pi_{/G}(v) ) , \pi_{/G}(w) \rb.
  \]
    We finally define the total force on our system to be
    \begin{align}
    	F = F_{\mu} + (F_\body \circ T \pi^Q_B), \label{eq:total_force}
   \end{align}
 where $F_\body$ is the dissipative force on the shape of the body mentioned in \S \ref{sec:passive solids}.  This total force $F$ descends via $\pi_{/G}$ to a reduced force $F_{/G} : P \to P^{\ast}$ where $P^*$ is the dual vector bundle to $P$.  The reduced force is given explicitly in terms of $f_\mu$ and $F_\body$ by $F_{/G} = f_\mu + (F_\body \circ \rho)$.  One can verify directly from this expression that $\lb F( v) , w \rb = \lb F_{/G}( \pi_{/G}(v) ) , \pi_{/G}(w) \rb$.
    
    We now introduce a consequence which follows from the $G$-invariance of $F$ and $L$.
  \begin{prop} \label{prop:evolution1}
    Let $X_{TQ} \in \mathfrak{X}(TQ)$ denote the Lagrange-d'Alembert vector field, and let $\Phi_t^{X_{TQ}} : TQ \to TQ$ denote the flow map associated with the Lagrangian $L:TQ \to \mathbb{R}$ and the force $F: TQ \to T^{\ast}Q$.  Then there exists a vector field $X_P \in \mathfrak{X}(P)$ and a flow map $\Phi_t^{X_P} : P \to P$  which are $\pi^{TQ}_{/G}$-related to $X_{TQ}$ and $\Phi_t^{X_{TQ}}$.
    \end{prop}
  
  \begin{proof}
    Let $q:[0,t] \to Q$ be a curve such that the time derivative $(q,\dot{q}): [0,t] \to TQ$ is an integral curve of the Lagrange-d'Alembert equations with initial condition $(q,\dot{q})(0)$ and final condition $(q,\dot{q})(t)$.  Then the Lagrange-d'Alembert variational principle states that
    \[
    	\delta \int_{0}^{t}{ L( (q,\dot{q})(\tau)) d\tau } = \int_{0}^{t}{ \lb F((q,\dot{q})(\tau) , \delta q(\tau) \rb d\tau}
    \]
    for all variations of the curve $q( \cdot)$ with fixed endpoints.  Note that for each $\psi \in G$, the action satisfies $\int_{0}^{t}{ L((q,\dot{q})(\tau) d\tau} = \int_{0}^{t}{ L( (q,\dot{q})(\tau) \circ \psi ) d\tau}$, and the variation on the right hand side of the Lagrange-d'Alembert principle is
   \begin{align*}
    	\int_{0}^{t}{ \lb F((q,\dot{q})(\tau)) , \delta q (\tau) \rb d\tau} &= \int_{0}^{t}{ \lb F_{/G}( \pi_{/G}( (q,\dot{q})(\tau) , \pi_{/G}( \delta q (\tau)) \rb d\tau} \\
	&= \int_0^t{ \lb F( (q,\dot{q})(\tau) \circ \psi ) , \delta q(\tau) \circ \psi \rb d\tau}.
   \end{align*}
   Therefore, we observe that
   \[
    	\delta \int_{0}^{t}{ L( (q,\dot{q})(\tau) \circ \psi) d\tau } = \int_{0}^{t}{ \lb F((q,\dot{q}) (\tau) \circ \psi) , \delta q (\tau) \circ \psi \rb d\tau}
   \]
   for arbitrary variations of the curve $q( \cdot )$ with fixed end points.  However, the variation $\delta q \circ \psi$ is merely a variation of the curve $q \circ \psi (\cdot)$ becuase
   \[
   	\delta q(\tau) \circ \psi = \left. \pder{}{\epsilon} \right|_{\epsilon = 0}( q(\tau, \epsilon) \circ \psi ),
   \]
   and if $q(\tau,\epsilon)$ is a deformation of $q( \tau)$, then $q(\tau,\epsilon) \circ \psi$ is a deformation of $q(\tau) \circ \psi$ by construction.
     Therefore,
   \[
    	\delta \int_{0}^{t}{ L( (q,\dot{q})(\tau) \circ \psi) d\tau} = \int_{0}^{t}{ \lb F((q,\dot{q})(\tau) \circ \psi) , \delta (q \circ \psi) \rb d\tau}
   \]
   for arbitrary variations of the curve $q \circ \psi$ with fixed end points.
   This last equation states that the curve $(q,\dot{q}) \circ \psi$ satisfies the Lagrange-d'Alembert principle.
   Thus, the flow $\Phi_t^{X_{TQ}}$ is $G$-invariant, as is the vector field $X_{TQ}$.
   By Proposition \ref{prop:reduced_vf}, there exists a $\pi^{TQ}_{/G}$-related flow and vector field on $TQ/G$.
   By Proposition \ref{prop:TQ/G}, we obtained the desired vector field $X_{P} \in \mathfrak{X}(P)$, and its flow $\Phi_t^{X_{P}}:P \to P$.
  \end{proof}
  
  Now that we know there exists a flow on $P$, one can ask for the equations of motion.
  
  \begin{prop}
  	The flow map of $P$  mentioned in Proposition \ref{prop:evolution1} for the Lagrangian $L$ and force $F$ is identical to the flow of the Lagrange-Poincar\'e-d'Alembert equation:
	\begin{align*}
		u_t + u \cdot \nabla u &= - \nabla p - \nu \Delta u \\
		\frac{D\dot{q}}{Dt} + \nabla U(q)  &= \sharp (F(q,\dot{q}) + F_{\partial \body}),
	\end{align*}
	where $F_{\partial \body} : P \to T^* B$ is the force that the fluid exerts on the body in order to satisfy the no-slip boundary condition.
  \end{prop}
  \begin{proof}
  	This is Theorem 7.1(c) of \cite{JaVa2013} paired with \eqref{eq:LDA}.
	Roughly speaking, one can obtain these equations of motion by performing Lagrange-Poincar\'e-d'Alembert reduction following \cite{CeMaRa2001}.
	This involves choosing a principal connection $A: TQ \to \mathfrak{g}$.
	The spatial velocity field is reconstructed by $u = h^{\uparrow}(b , \dot{b} , \varphi) + \varphi_*A( b, \dot{b} ,\varphi , \dot{\varphi})$,
	where $h^{\uparrow}$ is the horizontal lift.
  \end{proof}

\subsection{Reduction by frame invariance} \label{sec:SE}
  Consider the group of isometries of $\mathbb{R}^d$ denoted $\SE(d)$.
  Each $z \in \SE(d)$ sends $(b,\dot{b}, u) \in TQ/G$ to $(z (b,\dot{b}) , z_* u ) \in P$, where $z_*u \in \mathfrak{X}_{\rm div}( \bulk_{z \circ b})$ is the push-forward of the fluid velocity field $u \in \mathfrak{X}(\bulk_{b})$.
  This action is free and proper on $P$ so that the projection $\pi_{[P]}^{P} : P \to [P]$, where $[P] := \frac{P}{\SE(d)}$ is a principal bundle \cite[Prop 4.1.23]{FOM}.
  Additionally, $\SE(d)$ acts by vector-bundle morphisms, which are isomorphisms on each fiber.
  Therefore, $[P]$ inherits a vector-bundle structure from $P$.
   \begin{prop}
     There exists a unique vector-bundle projection $[\tau]: [P] \to [B]$ and a map $[\rho] : [P] \to T [B]$ such that $[\tau] \circ \pi^P_{[P]} = \tau$ and $[\rho] \circ \pi^{P}_{[P]} = \rho$.
   \end{prop}
   \begin{proof}
     We see that $\tau( z \cdot \xi ) = z \cdot \tau( \xi)$ for any $z \in \SE(d)$ and $\xi \in P$.
     Applying the above formula to an entire coset in $[P]$ then maps to a coset in $[B]$.
     Thus the map $[\tau]: [P] \to [B]$ is well-defined by the condition $\tau = [\tau] \circ [\cdot]$.
     The same argument makes the map $[\rho]: [P] \to T[B ]$.
     The vector-bundle structure on $[P]$ can be observed directly.
      \end{proof}
      
      As everything in sight is $\SE(d)$-invariant, it is not surprising that we can express reduced equations of motion on $[P]$.
         
   \begin{prop} \label{prop:evolution2}
     There exists a vector field $X_{[P]} \in \mathfrak{X}([P])$ and a flow-map $\Phi_t^{[P]}:[P] \to [P]$ which is $\pi^{P}_{[P]}$-related to $X_P$ and $\Phi_t^{X_P}$.
   \end{prop}
  \begin{proof}
    Let $q:[0,T] \to Q$ be a curve such that the time derivative $(q,\dot{q}): [0,T] \to TQ$ is an integral curve of the Lagrange-d'Alembert equations with initial condition $(q,\dot{q})_0$ and final condition $(q,\dot{q})_T$.  Then the Lagrange-d'Alembert variational principle states that
    \[
    	\delta \int_{0}^{T}{ L( q,\dot{q}) dt } = \int_{0}^{T}{ \lb F(q,\dot{q}) , \delta q \rb dt}
    \]
    for all variations of the curve $q( \cdot)$ with fixed endpoints.  Note that for each $\psi \in G$ and $z \in \SE(d)$, the action satisfies $\int_{0}^{T}{ L(q,\dot{q}) dt} = \int_{0}^{T}{ L( z \cdot (q,\dot{q}) \circ \psi ) dt}$ and the virtual-work is
   \begin{align*}
    	\int_{0}^{T}{ \lb F(q,\dot{q}) , \delta q \rb dt} &= \int_{0}^{T}{ \lb F_{/G}(z \cdot  \pi_{/G}( q,\dot{q}) , z \cdot \pi_{/G}( \delta q) \rb dt} \\
	&= \int_0^T{ \lb F( z \cdot (q,\dot{q}) \circ \psi ) , z \cdot \delta q \circ \psi \rb dt}.
   \end{align*}
   Therefore we observe that
   \[
    	\delta \int_{0}^{T}{ L( z \cdot (q,\dot{q}) \circ \psi) dt } = \int_{0}^{T}{ \lb F( z \cdot (q,\dot{q}) \circ \psi) , z \cdot \delta q \circ \psi \rb dt}
   \]
   for arbitrary variations of the curve $q( \cdot )$ with fixed end points.  However, the variation $z \cdot \delta q \circ \psi$ is merely a variation of the curve $z \cdot q \circ \psi (\cdot)$.  Therefore,
   \[
    	\delta \int_{0}^{T}{ L( z \cdot (q,\dot{q}) \circ \psi) dt } = \int_{0}^{T}{ \lb F( z \cdot (q,\dot{q}) \circ \psi) , \delta (z \cdot q \circ \psi) \rb dt}
   \]
   for arbitrary variations of the curve $z \cdot q \circ \psi$ with fixed end points.  This last equation states that the curve $z \cdot (q,\dot{q}) \circ \psi$ satisfies the Lagrange-d'Alembert principle.  Thus, $\Phi_T^{TQ}(z \cdot (q,\dot{q})_0 \circ \psi)= z \cdot \Phi_{T}^{TQ}(q,\dot{q}) \circ \psi$.  Since $\psi \in G$ was arbitrary, we may apply $\Phi_T^{TQ}$ to the entire coset $z \cdot (q,\dot{q}) \cdot G$ to find $\Phi^{TQ}_T( z \cdot (q,\dot{q} ) \cdot G) = z \cdot \Phi_T^{TQ}(q,\dot{q}) \cdot G$.
  This map from $G$-cosets to $G$-cosets is the defining condition for $\Phi_T^P$.
  Therefore, the last equation states that $\Phi^P_T(z \cdot \xi) = z \cdot \Phi_T^P(\xi)$, where $\xi = \pi^{TQ}_{/G}((q,\dot{q})_0)$.
  In other words, $\Phi_T^P$ is $\SE(d)$-invariant.
  Therefore, by Proposition \ref{prop:reduced_vf} the theorem follows.
  \end{proof}

\section{Asymptotic Behavior} \label{sec:asy}
  It is commonplace to assume that the asymptotic behavior of a simple mechanical system with dissipation approaches a state of minimum energy.
  In this section, we will verify that the asymptotic behavior of the Lagrangian system described in the previous section tends towards the minimizers of the elastic potential energy, $U$.

\begin{prop} \label{prop:stable_set}
Assume the Lagrangian $L$ of \eqref{eq:total_Lagrangian} and the external force $F$ of \eqref{eq:total_force}.  Let $q:[0,\infty) \to Q$ be a curve such that the time derivative $(q,\dot{q}) : [0, \infty) \to TQ$ is an integral curve of the Lagrange-d'Alembert equations for the Lagrangian $L$ and the force $F$.  Then the $\omega$-limit set of $(q,\dot{q})( \cdot )$ is contained in the set $dU^{-1}(0) := \{ (q,0) \in TQ \quad \vert \quad dU(q) = 0 \}$.
\end{prop}
\begin{proof}
  The energy is the function $E: TQ \to \mathbb{R}$ given by
  \begin{align*}
  	E(q,\dot{q}) &:= \lb \mathbb{F}L(q,\dot{q}) , \dot{q} \rb - L(q,\dot{q}).
  \end{align*}
   Given any Lagrangian system on a Riemannian manifold where the Lagrangian is the kinetic energy minus the potential energy, the time derivative of the generalized energy under the evolution of the Lagrange-d'Alembert equations is given by $\dot{E} = \lb  F(\dot{q}) , \dot{q} \rb$.  In this case, we find
   \[
   	\dot{E}(q,\dot{q}) = \lb F_{\body}(\dot{b}) , \dot{b} \rb + \lb F_{\mu}( q,\dot{q}), (q,\dot{q}) \rb.
   \]
   However, by \eqref{eq:concavity}, this is a convex function on each fiber of $TQ$ in a neighborhood of the zero-section.
   Therefore the $\omega$-limit of $(q,\dot{q})( \cdot )$, denoted $M^{\omega}$, must be a subset of the zero section of $TQ$.
   Moreover, the Lagrange-D'Alembert equations state
   \[
   	\frac{D \dot{q}}{Dt} = - \nabla U(b) + \sharp( F(q,\dot{q}) ),
   \]
   where $\sharp:T^{\ast}Q \to TQ$ is the sharp map associated the metric on $Q$.  However, $F(q,\dot{q}) = 0$ when $\dot{q} = 0$, which is the case for points in $M^{\omega}$.  Thus, the vector field on $M^{\omega}$ must satisfy
   \[
   	\frac{D \dot{q}}{Dt} = - \nabla U(b).
   \]
   However, $M^\omega$ is an invariant set.
   Therefore, the Lagrange-d'Alembert vector field must be tangential to $M^\omega$.
   As $M^\omega$ is contained in the zero-section of $TQ$, the second derivative of $q(t)$ must vanish in order to remain in the zero section.
   Thus, we find $\frac{D \dot{q}}{Dt} = 0$ on $M^{\omega}$ which implies $\nabla U = 0$ on $M^{\omega}$.
   That is to say, $M^{\omega} \subset dU^{-1}(0)$.
\end{proof}

\begin{cor} \label{cor:stable_point}
  Let $[U]: [B] \to \mathbb{R}$ be the unique function on the shape-space of the body such that $[U]( [ b] ) = U(b)$ for all $b \in B$.  Assume that $[U]$ has a unique minimizer $s_{\min} \in [B]$, and let $(s_{\min})^{0}_{\uparrow} \in [P]$ denote the element of the zero section of $[\tau]: [P] \to [B]$ above $s_{\min} \in [B]$.  If $(q,\dot{q}): [0,\infty) \to TQ$ is an integral curve of the Lagrange-d'Alembert equations, then $[\xi] (t) = [ \pi_{/G}(q,\dot{q}(t))]$ must approach $(s_{\min})^{0}_{\uparrow} \in [P]$.  If the flow of the Lagrange-d'Alembert equations is complete, this means that $(s_{\min})^0_{\uparrow}$ is a global (weakly) hyperbolically stable fixed point for the vector field $X_{[P]}$.
\end{cor}
\begin{proof}
  In proposition \ref{prop:stable_set}, we showed that solutions approach points within the set $dU^{-1}(0)$ asymptotically.  This implies that the dynamics on $[P]$ must approach $d[U]^{-1}(0)$ asymptotically.  However, there is only one such point.
\end{proof}

 In the next section, we will periodically perturb this stable equilibria to obtain a loop in $[P]$.

\begin{example}
  Consider the swimmer of Example \ref{ex:two_link} and Figure \ref{fig:two_link}.  Corollary \ref{cor:stable_point} asserts that the state where the swimmer and the water is stationary and $\phi_1 - \phi_2 = \bar{\theta}$ is asymptotically stable.
\end{example}

\section{Swimming} \label{sec:swimming}
%
  In order to understand swimming, let us consider a time-periodic internal body force, $F_{\rm swim}: T [B] \times S^1 \to T^{\ast}[B]$.
  Such a force should be designed to model the time-periodic activation of muscles in a fish, or control forces for an underwater vehicle.
  This force can be lifted via the map $\pi^{B}_{[B]} \circ \pi^Q_B: Q \to [B]$ to obtain a $G,\SE(d)$-invariant invariant force on $Q$.
  \emph{The addition of this time periodic force force on $Q$ alters the Lagrange-d'Alembert equations linearly by the addition of a $G,\SE(d)$-invariant time-periodic vector field $X_{\rm swim}$.}
  To consider small perturbations, we can consider scaling this time-periodic force by a real parameter $\varepsilon \in \mathbb{R}^+$, so that the Lagrange-d'Alembert vector field is now given by the time-periodic vector field $X_{TQ,\varepsilon} = X_{TQ} + \varepsilon X_{\rm swim} \in \mathfrak{X}(TQ)^{S^1}$.
  This deformed vector-field is also $G,\SE(d)$-invariant; thus, there exists a reduced vector fields $X_{P,\varepsilon} \in \mathfrak{X}(P)^{S^1}$ which is $\pi^Q_{P}$-related to $X_{TQ,\epsilon}$,  and a vector field $X_{[P],\varepsilon} \in \mathfrak{X}([P])^{S^1}$ which is $\pi^{P}_{[P]}$-related to $X_{P,\varepsilon}$.
 The vector field $X_{[P],\varepsilon}$ is a time-periodic deformation of $X_{[P]}$.
 As $X_{[P]}$ admits an asymptotically stable point by Corollary \ref{cor:stable_point}, we are reasonably close to being able to prove the existence of a $\SE(d)$-relative limit cycle for $X_{P,\varepsilon}$ for small $\varepsilon > 0$.
 
 {\bf Desideratum:}
 {\it For sufficiently small $\varepsilon > 0$, the vector-field $X_{[P],\varepsilon}$ admits a non-autonomous exponentially stable limit cycle.
 Moreover, $X_{P,\varepsilon}$ admits stable relative limit cycle. }
 
 If we were to assume Proposition \ref{prop:NAESRLC} held for infinite-dimensional dynamical systems, then as $X_{[P],\varepsilon}$ is a deformation of $X_{[P]}$ we could deduce that $X_{P,\varepsilon}$ admits a non-autonomous exponentially stable limit cycle for sufficiently small $\varepsilon > 0$.
 As a result, this would imply $X_{P,\epsilon}$ admits $\SE(d)$-relative limit cycles which are $\pi^{P}_{[P]}$-related to the limit cycle in $[P]$.

 Unfortunately, we are unable to do this because Proposition \ref{prop:NAESRLC} is limited to finite-dimensional manifolds and vector fields with exponentially stable points.
 The vector-field $X_{[P]}$ is on an infinite-dimensional space where we have only proven asymptotic stability.
 We will not overcome this difficulty; however, we are at least able to speculate on how to deal with this.
 For example, there does exist extensions of normal hyperbolicity and persistence to infinite-dimensional dynamical systems on a Hilbert manifold  \cite{Jones1999,Bates1998}.
 Alternatively, we could consider finite-dimensional models for fluid-structure interaction, such as the immersed boundary method \cite{Peskin2002}.
 In the next section, we will informally speculate on this latter approach.

\subsection{Analytic concerns and approximate relative limit cycles} \label{sec:analytic}
Up until now, the paper has been fairly rigorous and complete.
This start of this section marks the end of this theorem-proof formalism.
Instead, we provide a more speculative discussion on how one can overcome the challenges to obtaining relative limit cycles in $P$.

There are two issues of concern.  The first is the lack of a ``spectral gap'' with respect to the equilibrium associated to $s_{\min} \in [B]$.
That is to say, it is not immediately obvious if there exists a convergence bound $\rho > 0$ with respect to $s_{\min}$, as is required in order to use Theorem \ref{thm:persistence} and its offspring, Proposition \ref{prop:NAESRLC}.
It is possible that there does not exist any such $\rho$.
For simple mechanical systems, $\rho$ is related to the spectrum of the Rayleigh dissipation function.
In our case, this spectrum includes the spectrum of the Laplace-operator on a non-compact domain, which \emph{does not contain a spectral gap!}

The second issue is the non-completeness of $[P]$.  As $Q$ is an infinite-dimensional Fr\'echet manifold, so is $[P]$.
This is a concern because both Theorem \ref{thm:persistence} and Proposition \ref{prop:NAESRLC} require completeness in order to provide an existence-uniqueness result.
There do exist generalizations of Theorem \ref{thm:persistence} to infinite-dimensional Banach manifolds, but not Fr\'echet manifolds \cite{Jones1999,Bates1998}.

  Therefore, using the persistence theorem directly will not allow us to assert the existence of a relative limit cycle on $P$.
  Perhaps other methods besides normal hyperbolicity theory could be employed, but this would be an exploration for another paper.
  
    However,  we can consider an option which is morally the converse of an idea illustrated in \cite{Jones1999}, wherein discrete approximations are invoked.
    There exists a number of finite-dimensional models for the space $P$ used by engineers to study fluid-structure interaction.
    It is fairly common to approximate the fluid velocity field on a finite-dimensional space and model the solid using a finite element method (e.g. \cite{Peskin2002}).
    Let us call this finite-dimensional space $P_{\discrete}$.
    Moreover, one can usually act on $P_{\discrete}$ by $\SE(d)$ by simply rotating and translating the finite elements and the grid.
    If the model on $P_{\discrete}$ converges as the time step and spatial resolution go to zero, then we could reasonably restrict ourselves to methods which dissipate energy at a rate which is quadratic and positive definite in the state velocity.
    This is not too much to expect, as a good method ought to converge.\footnote{The immersed boundary method \cite{Peskin2002} and smooth-particle hydrodynamics \cite{Monaghan1977,Lucy1977} are both candidates.}
    By the same arguments as before, the dynamics will exhibit hyperbolically stable equilibria on the quotient space $[P_{\discrete}] = \frac{ P_{\discrete} }{\SE(d)}$.
    Upon adding a periodic perturbation to the dynamics on $[P_{\discrete}]$, one could apply Proposition \ref{prop:NAESRLC} directly to assert the existence of a non-autonomous exponentially stable relative limit cycle $\gamma_{\discrete}(t) \in P_{\discrete}$.  In particular, by Proposition \ref{prop:regular}, $\gamma_{\discrete}(t)$ must satisfy
    \[
	\gamma_{\discrete}(t) = z^{ \lfloor t \rfloor } \cdot \gamma_{\discrete}(t - \lfloor t \rfloor)
    \]
 for some $z \in \SE(d)$, where $\lfloor t \rfloor = \sup \{ k \in \mathbb{Z} : k \leq t \}$.
    If the model on $P_{\discrete}$ converges, then there exists a trajectory in $\gamma(s) \in P$ which is well-approximated by $\gamma_{\discrete}(t)$, over a single time period, by the definition of ``convergence.''
    Then, the equation $\gamma(t) = z^{ \lfloor t \rfloor } \cdot \gamma(t - \lfloor t \rfloor)$ would hold up to numerical error.
    In other words, the immersed body would move in an \emph{approximately} relatively periodic fashion, reminiscent of swimming.

\section{Conclusion and Future work}  \label{sec:Conclusion}
  It is widely observed that steady swimming is periodic, and this observation inspired the question, ``Is it possible to interpret swimming as a limit cycle?''
  In this paper, we have illustrated the crucial role played by $\SE(d)$-reduction in answering this question.
  Moreover, we have posed a possible answer, accurate up to the spatial discretization error of a numerical method.
  The existence of these hypothetical relative limit cycles would provide robustness to mechanisms of locomotion, and conform with behavior observed in real systems \cite{Alben_Shelley_2005,Bhalla2013,Liao2003a,Liao2003b,Tytell2010,Wilson2011}.
  \emph{Given the complexity of fluid-structure interaction, it is not immediately clear that one could expect such orderly behavior.}
  This potential orderliness could be exploited in a number of applications.

\begin{enumerate}
\item {\bf Robotics and Optimal Control}
  The interpretation of swimming as a limit cycle may permit a non-traditional framework for controller design.
  For example, if our control forces are parametrized by a space $C$, then we may consider the set of loops, $\mathrm{loop}(C)$.
  The limit cycle hypothesis would imply the existence of a subset $W \subset {\rm loop}(C)$ and a map $\Gamma: W \to \mathrm{loop}( [P] )$ which outputs the periodic limit cycle in $[P]$, resulting asymptotically from the time-periodic control signals in $W$.  Given $\Gamma$,  we may define a control cost functional on $W$ based upon a reward function on ${\rm loop}([P])$.  As such a cost functional would only respond to the asymptotic behavior of the system, one could surmise that it would not overreact to transient dynamics.
 
\item {\bf Transient dynamics} Although trajectories may approach a limit cycle, the transient dynamics are still important.
The transient dynamics would re-orient and translate the body before orderly periodic behavior takes effect.
Therefore, if one desires to create locomotion through periodic control inputs, one should try to get onto a limit cycle quickly in order to minimize the duration where transient dynamics dominate.

\item {\bf Pumping}
  In the current setup, one could consider a reference frame attached to the body.  In this reference frame, ``swimming'' manifests as fluid moving around the body in a regular fashion.  This change in our frame of reference describes pumping.

\item {\bf Passive Dynamics}
  This paper does \emph{not} address the dual problem.  By the dual problem, we mean: ``Given a constant fluid velocity at infinity, what periodic motion (if any) will a tethered body approach as time goes to infinity?''  In this dual problem, the motion of the body is given first, and parameters such as the period of the limit cycle are emergent phenomena.  In particular, the dual problem of a flapping flag immersed in a fluid with a constant velocity at infinity has received much attention in the applied mathematics community (see \cite{Shelley2005} and references therein).  Here, it is generally not the case that a limit cycle will emerge, and the system is capable of admitting chaos.
\item {\bf Other types of locomotion}
  The notion that walking may be viewed as a limit cycle is fairly common \cite{GaChaRuCo98,Hobbelen,McGreer1990}.  Moreover, it is conceivable that flapping flight is a limit cycle as well \cite{LiuRistroph2012}.  However, for both of these systems, $\SE(3)$ symmetry is broken by the direction of gravity.
    Because of this, it is not immediately clear that one can import the methods used here to understand flapping flight and terrestrial locomotion.  However, perhaps this is merely a challenge to be overcome.  In particular, these systems still exhibit $\SE(2)$ symmetry.  For the case of 2D bipedal walkers, we have an $\mathbb{R}$-symmetry and the stability problems due to falling will not manifest.  Here, one can find limit cycles using regularized models of the ground \cite{JaEl2013}.
\end{enumerate}

\subsection{Acknowledgements}
The notion of swimming as a limit cycle was initially introduced to me by Erica J. Kim while she was studying hummingbirds.
Additionally, Sam Burden, Ram Vasudevan, and Humberto Gonzales provided much insight into how to frame this work for engineers.
I would also like to thank Professor Shankar Sastry for allowing me to stay in his lab for a year and meet people who are outside of my normal research circle.
I would like to thank Eric Tytell for suggesting relevant articles in neurobiology,
Amneet Pal Singh Bhalla for allowing me to reproduce figures from \cite{Bhalla2013}, and Peter Wallen for allowing me to reproduce figures from \cite{WallenWilliams1984}.
An early version of this paper was written in the context of Lie groupoid theory, where the guidance of Alan Weinstein was invaluable.
Jaap Eldering and Joris Vankerschaver have given me more patience than I may deserve by reading my papers and checking my claims.
Major contributions to the bibliography and the overall presentation of the paper were provided by Jair Koiller.
Finally, the writing of this paper was solidified with the help of Darryl Holm.
This research has been supported by the European Research Council Advanced Grant 267382 FCCA and NSF grant CCF-1011944.

\bibliographystyle{amsalpha}
\bibliography{hoj_swimming}

\end{document}